\documentclass{amsart} 
\usepackage{amsmath,amssymb,amsthm,amsfonts,amsxtra,mathtools}
\usepackage{enumerate,verbatim,mathrsfs,comment,color}
\usepackage[all,2cell,ps]{xy}
\usepackage[pagebackref]{hyperref}
\usepackage{graphicx}
\usepackage{tikz-cd}  
\usepackage{verbatim}

\DeclareMathOperator{\Id}{Id} 
\DeclareMathOperator{\ann}{ann}
\DeclareMathOperator{\Ass}{Ass}
\DeclareMathOperator{\Char}{char}
\DeclareMathOperator{\cid}{CI-dim}

\DeclareMathOperator{\curv}{curv}
\DeclareMathOperator{\cx}{cx}
\DeclareMathOperator{\depth}{depth}
\DeclareMathOperator{\edim}{edim}

\DeclareMathOperator{\embdim}{edim}

\DeclareMathOperator{\Ext}{Ext}

\DeclareMathOperator{\gdim}{G-dim}

\DeclareMathOperator{\Hom}{Hom}

\DeclareMathOperator{\id}{id}

\DeclareMathOperator{\Min}{Min}
\DeclareMathOperator{\Mod}{mod}

\DeclareMathOperator{\pd}{pd}

\DeclareMathOperator{\rank}{rank}

\DeclareMathOperator{\Soc}{Soc}

\DeclareMathOperator{\syz}{\Omega}
  
\DeclareMathOperator{\Tor}{Tor}

\renewcommand{\ge}{\geqslant}
\renewcommand{\le}{\leqslant}

\newcommand{\fm}{\mathfrak{m}}
\newcommand{\fp}{\mathfrak{p}}
\newcommand{\fq}{\mathfrak{q}}
\renewcommand{\iff}{if and only if }

\theoremstyle{plain}
\newtheorem{theorem}{Theorem}[section]
\newtheorem{lemma}[theorem]{Lemma}
\newtheorem{proposition}[theorem]{Proposition}
\newtheorem{corollary}[theorem]{Corollary}
\newtheorem{conjecture}[theorem]{Conjecture}

\theoremstyle{definition}
\newtheorem{definition}[theorem]{Definition}

\newtheorem{example}[theorem]{Example}

\newtheorem{para}[theorem]{}

\newtheorem{setup}[theorem]{Setup}

\theoremstyle{remark}
\newtheorem{remark}[theorem]{Remark}

\numberwithin{equation}{section}

\title[Complexity and rigidity of Ulrich modules]{Complexity and rigidity of Ulrich modules, and some applications}

\author{Souvik Dey}
\address{Department of Mathematics, University of Kansas, 405 Snow Hall, 1460 Jayhawk Blvd, Lawrence, KS 66045, U.S.A.}
\email{souvik@ku.edu} 

\author[Dipankar Ghosh]{Dipankar Ghosh}
\address{Department of Mathematics, Indian Institute of Technology Kharagpur, Kharagpur - 721302, West Bengal, India}
\email{dipankar@maths.iitkgp.ac.in}

\date{}       

\begin{document}

\pagenumbering{arabic}
\thispagestyle{empty} 
  
 \begin{abstract}
 	We analyze whether Ulrich modules, not necessarily maximal CM (Cohen-Macaulay), can be used as test modules, which detect finite homological dimensions of modules. We prove that Ulrich modules over CM local rings have maximal complexity and curvature. Various new characterizations of local rings are provided in terms of Ulrich modules. We show that every Ulrich module of dimension $s$ over a local ring is $(s+1)$-Tor-rigid-test, but not $s$-$\Tor$-rigid in general (where $s\ge 1$). Over a deformation of a CM local ring of minimal multiplicity, we also study Tor rigidity.  
 \end{abstract}

\maketitle

\section{Introduction}
	
	\begin{setup}\label{setup}
		Unless otherwise specified, let $(R,\fm,k)$ be a commutative Noetherian local ring of dimension $d$. All $R$-modules are assumed to be finitely generated unless otherwise stated.
	\end{setup}  
	
	The notion of Ulrich modules that we use in the present article is as in \cite[Defn.~2.1]{GTT15}, namely, a non-zero CM (Cohen-Macaulay) $R$-module is called Ulrich if $e(M)=\mu(M)$, where $e(M)$ denotes the multiplicity of $M$ wth respect to $\fm$, and $\mu(M)$ is the minimal number of generators of $M$. When $R$ is CM and $M$ is MCM (maximal Cohen-Macaulay), this notion coincides with the definition of \cite[pp.~183]{BHU87}. When $k$ is infinite, then $M$ is Ulrich if and only if $\fm M = ({\bf x})M$ for some system of parameters ${\bf x}=x_1,\cdots,x_s$ of $M$; see \cite[Prop.~2.2.(2)]{GTT15}. 
	
	In this article, we analyze whether Ulrich modules can be used as test modules, which detect finite homological dimensions of modules (particularly, freeness, or projective dimension and injective dimension of a module), and we characterize various local rings using Ulrich modules. We study complexity and rigidity of Ulrich modules, and provide some applications. We prove the following results on rigidity in Section~\ref{sec:Rigidity of Ulrich modules}. See Definition~\ref{defn:test-modules} for the terminologies on Tor rigidity and test modules.  
	
	\begin{theorem}[See \ref{cor:Ulrich-mod-d+1-Tor-rigid} and \ref{thm:exam-reduced-Gor-dim-1-min-mult}]
		Let $M$ be an Ulrich $R$-module of dimension $s$. Then
		\begin{enumerate}[\rm (1)]
			\item $M$ is {\rm $(s+1)$-$\Tor$-rigid-test}.
			\item $M$ need not be $s$-$\Tor$-rigid. Indeed, for every reduced Gorenstein local ring of dimension $1$ and of minimal multiplicity which is not an integral domain, there exists MCM Ulrich modules that are neither rigid-test nor strongly rigid.
		\end{enumerate}
	\end{theorem}
	
	In \cite[Cor.~9]{Avr96}, Avramov proved that every nonzero homomorphic image of a finite direct sum of syzygy modules of the residue field has maximal complexity and curvature (see Definitions~\ref{def:complexity} and \ref{defn:max-complexity-curvature}). In Section~\ref{sec:Complexity of Ulrich modules}, we show that Ulrich modules also have maximal complexity and curvature, see Theorem~\ref{thm:Ulrich-max-cx}. Moreover, we obtain a number of characterizations of various CM local rings via Ulrich modules. We prove the following results in this regard.
	
	\begin{theorem}[See \ref{cor:charac-CI-via-Ulrich} and \ref{prop:charac-regular-by-Ulrich-mod}]\label{thm:complexity-various-charac-local-rings}
		Let $M$ be an Ulrich $R$-module.
		\begin{enumerate}[\rm (1)]
			\item When $R$ is CM, then
			\begin{center}
				$R$ is complete intersection $\Longleftrightarrow$ $\cx_R(M) < \infty$ $\Longleftrightarrow$ $\curv_R(M) \le 1$.
			\end{center}
			\item If $N$ is also an Ulrich $R$-module, then
			\begin{align*}   
			\mbox{$R$ is regular} \;&\Leftrightarrow \; \Ext_R^{n \le i \le n+s}(M,N)=0 \mbox{ for some } n\ge 1\text{ with } n+\dim(M)\ge \depth_R N\\
			& \Leftrightarrow \; \Tor^R_{n \le i \le n+s}(M,N)=0 \; \mbox{ for some }n\ge 1,
			\end{align*}
			where $s=\min\{\dim(M),\dim(N)\}$.
		\end{enumerate}
	\end{theorem}
	
	Similar characterizations as in Theorem~\ref{thm:complexity-various-charac-local-rings}.(2) of Gorenstein local rings are shown in terms of certain consecutive vanishing of $\Ext$ or $\Tor$ involving an Ulrich module and a nonzero module of finite projective or injective dimension, see Proposition~\ref{prop:charac-Gor-by-Ulrich-mod}.
	
	Next, we analyze MCM Ulrich modules over a CM local ring of minimal multiplicity, and provide some applications.
	
	\begin{theorem}[See \ref{cor:min-mult-comlete-inter}, \ref{prop:exam-Ulrich-hom-image-syz-k}, \ref{minrigid} and \ref{cor:ARC-min-mult}]\label{thm:results-over-min-mult}
		Suppose that $R$ is CM of minimal multiplicity.
		\begin{enumerate}[\rm (1)]
			\item Then $ R $ is an abstract hypersurface \iff $R$ admits a module $M$ such that $ 0 < \cx_R(M) < \infty $ $($\iff $ 0 < \curv_R(M) \le 1 $$)$.
			\item Every MCM homomorphic image of a finite direct sum of syzygy modules of $k$ is Ulrich.
			\item Every $R$-module $M$ is $ (2d-t+2) $-$ \Tor $-rigid, where $ t = \depth_R(M) $.
			\item For an $R$-module $M$ with $t = \depth(M)$, if
				$\Ext^i_R(M,M)=\Ext^j_R(M, R)=0$ for all $1\le i \le d+1$ and $1\le j \le 2d-t+2$,
			then $M$ is free.
		\end{enumerate}
	\end{theorem}

	Next, we strengthen \cite[Cor.~6.5]{GP19b} (one of the two main results in that paper). See Corollary~\ref{deformcor} and the preceding paragraph. Moreover, we analyze the rigidity of an arbitrary module over a deformation $S$ of a CM local ring of minimal multiplicity. In Section~\ref{sec:Ext-persistency-ARC}, we study a form of Ext-persistency of $S$. We also show that the ARC (Auslander-Reiten Conjecture \ref{ARC}) \cite{AR75} holds true over such rings, which recovers a very special case of \cite[Cor.~7.3]{AINS20}. We summarize the main results of the last two sections in the following theorem.
	
	\begin{theorem}[See \ref{deformcor}, \ref{gdeform}, \ref{cor:rigidity-over-deform}, \ref{thm:defor-Ext-persistency} and \ref{thm:ARC-over-deform}]\label{thm:results-over-deform-min-mult}
		Suppose $R$ is CM of minimal multiplicity. Let $ S = R/(f_1,\dots,f_c)R $, where $ f_1,\dots,f_c \in \fm $ is an $R$-regular sequence. Let $M$ and $N$ be $S$-modules. Set $i_0:=\dim(S)-\depth(M)$.
		\begin{enumerate}[\rm (1)]
			\item The following are equivalent:	
			\begin{enumerate}[\rm (i)]
				\item  $\Tor_i^S(M,N)=0$ for some $(d+c+1)$ consecutive values of $i\ge i_0+2$. 
				\item $\Tor_i^S(M,N)=0$ for all $i\ge i_0+1$.  
			\end{enumerate}	
			Moreover, if this holds true, then  either $ \pd_R M < \infty $ or $ \pd_R N < \infty $.
			
			If $N$ is embedded in an $S$-module of finite projective dimension, then the threshold $i_0+2$ above can be further improved to $i_0+1$.
			\item The $S$-module $M$ is $(d+c+i_0+2)$-$\Tor$-rigid.
			\item If $ i_0 \le 1$ and $\Ext^i_S(M,M)=0$ for some $(d+c+1)$ consecutive values of $i\ge i_0+2$, then either $\pd_S M < \infty $ or $\id_S M<\infty$.
			\item If $\Ext^i_S(M,M)=\Ext_S^j(M,S)=0$ for all $i_0+2 \le i \le i_0+(d+c)+2$ and $1 \le j \le 2i_0 + (d+c)+2$, then $M$ is free.
		\end{enumerate}
	\end{theorem}

%
%
%
%
%
%
%
\section{Rigidity of Ulrich modules}\label{sec:Rigidity of Ulrich modules}

In this section, we analyze the rigidity of Ulrich modules. We first recall the notion of rigidity.

\begin{definition}\label{defn:test-modules}{~}
	\begin{enumerate}[\rm (1)]
		\item
		Let $ m \ge 1 $ be an integer. An $R$-module $M$ is called $m$-Tor-rigid if for every $R$-module $N$, $\Tor_{n \le i < n+m}^R(M,N)=0$ for some $ n \ge 1 $ implies that $ \Tor_{i \ge n}^R(M,N) = 0 $, i.e., $m$ consecutive vanishing of Tor implies that the vanishing of all subsequent Tor. The $R$-module $M$ is called Tor-rigid if it is $1$-Tor-rigid \cite[Sec.~2]{Aus61}.
		\item
		\cite[Defn.~1.1]{CDT14} An $R$-module $M$ is called a test module if for every $R$-module $N$, $ \Tor_{i \gg 0}^R(M,N) = 0 $ implies that the projective dimension $ \pd_R(N) $ is finite.
		\item
		\cite[Defn.~2.3]{ZCMS18} An $R$-module $M$ is called rigid-test (resp., $m$-Tor-rigid-test) if it is both Tor-rigid (resp., $m$-Tor-rigid) and test module.
		\item
		\cite[Defn.~2.1]{DLM10} An $R$-module $M$ is called strongly rigid if for every $R$-module $N$, $\Tor_i^R(M,N)=0$ for some $i \ge 1$ implies that $\pd_R(N)< \infty$.
	\end{enumerate}
\end{definition}

\begin{example}\label{exam:test modules}{~}
	\begin{enumerate}[\rm (1)]
		\item \cite[Cor.~2.2]{Aus61}, \cite[Cor.~1 and Thm.~3]{Lic66} A celebrated result by Auslander and Lichtenbaum, known as rigidity theorem, says that modules over regular local rings, and modules of finite projective dimension over hypersurfaces are Tor-rigid.
		\item Let $I$ be an integrally closed $\fm$-primary ideal of $R$. Then $I$ is rigid-test as an $R$-module, see \cite[Cor.~3.3]{CHKV06}.
		\item \cite[Thm.~1.4]{CDT14} Over a complete intersection local ring, the test $R$-modules are precisely the $R$-modules of maximal complexity.
	\end{enumerate}
\end{example}

By definition, every rigid-test module is both Tor-rigid and strongly rigid module, while a Tor-rigid module need not be a rigid-test module, see \cite[Example~6.3]{ZCMS18}.

In order to study rigidity of Ulrich modules, we start with the following elementary lemma about consecutive vanishing of Ext and Tor.

\begin{lemma}\label{lem:cons-vanishing-Tor}
	Let ${\bf x} = x_1,\dots,x_t$ be an $M$-regular sequence. Let $m$ and $n$ be positive integers. Then the following statements hold true.
	\begin{enumerate}[\rm (1)]
		\item If $\Tor^R_i(M,N)=0$ for $n\le i\le m+n$, then $\Tor^R_i(M/{\bf x} M,N)=0$ for $n+t\le i\le m+n$.
		\item   If $\Ext^i_R(M,N)=0$ for $n\le i \le m+n$, then  $\Ext^i_R(M/{\bf x} M,N)$ for all $n+t\le i\le m+n$.		
		\item  If $\Ext^i(N,M)=0$ for all $n\le i\le m+n$, then $\Ext^i_R(N,M/{\bf x} M)=0$ for all $n\le i\le m+n-t$.   
	\end{enumerate} 
\end{lemma}

\begin{proof}
	It is enough to prove each of the claims when $t=1$. Let $x=x_1$. Consider the short exact sequence
	\begin{equation}\label{s.e.s}
		0 \longrightarrow M \stackrel{x \cdot}{\longrightarrow} M \longrightarrow M/xM \longrightarrow 0.
	\end{equation}
	
	(1) Applying $(-)\otimes_R N$ to \eqref{s.e.s}, there is an exact sequence
		$\Tor^R_{i}(M,N) \to \Tor^R_i(M/xM,N) \to \Tor^R_{i-1}(M,N)$
	for all $i\ge 0$. Hence it follows from the given vanishing condition that $\Tor^R_i(M/xM,N)=0$ for all $ n+1 \le i \le m+n $.
	
	(2) Applying $\Hom_R(-,N)$ to \eqref{s.e.s}, we get the exact sequence
		$\Ext^{i-1}_R(M,N) \to \Ext^i_R(M/xM,N) \to \Ext^i_R(M,N)$
	for all $i\ge 0$. Since $\Ext^i_R(M,N)=0$ for all $n\le i\le m+n$, from the exact sequence, we obtain that $\Ext^i_R(M/xM,N)$ for all $n+1\le i\le m+n$.
	
	(3) Applying $ \Hom_R(N,-) $ to \eqref{s.e.s}, there is an exact sequence $\Ext^i_R(N,M)\to \Ext^i_R(N,M/xM) \to \Ext^{i+1}_R(N,M)$ for all $i\ge 0$. Hence, since $ \Ext^i_R(N,M) = 0 $ for all $n\le i\le m+n$, it follows that $\Ext^i_R(N,M/xM)=0$ for all $n\le i\le m+n-1$.
\end{proof}

Next we show that Ulrich modules can be used as test modules, which detect the finiteness of projective dimension and injective dimension of arbitrary modules. Regarding the definition of Ulrich modules as in \cite[Defn.~2.1]{GTT15}, we give a quick proof of the standard inequality $e(M) \ge \mu(M)$ for any CM $R$-module $M$, which is probably known, but we are unable to find an explicit reference.

\begin{remark}\label{rmk:inequality}
	We may pass to a faithfully flat extension to assume that $R$ has infinite residue field. If $\dim M=0$, then $e(M)=\lambda_R(M)\ge \mu(M)$. So assume that $\dim M>0$. By \cite[4.6.10]{BH93}, $e(M)=e((\mathbf x),M)$ for some system of parameters $\mathbf x=x_1,...,x_s$ on $ M $, where $ s := \dim(M) $. Then, $\mathbf x$ is $M$-regular since $M$ is CM (\cite[2.1.2.(d)]{BH93}). Let $J=(\mathbf x)$.  Since $\mathbf x$ is an $M$-regular sequence, by \cite[1.1.8]{BH93}, we have
	\[
		J^nM/J^{n+1}M\cong (M/JM)^{\bigoplus \dfrac{(n+s-1)!}{n!(s-1)!}}.
	\]
	Hence the multiplicity of $M$ is given by
	\begin{align*}
	e(M) = e(J,M) &= (s-1)!\lim_{n\to\infty}\dfrac{\lambda_R(J^nM/J^{n+1}M)}{n^{s-1}} \\
	&= \lambda_R(M/JM) \ge \lambda_R(M/\fm M) = \mu(M).
	\end{align*}
\end{remark}

Proposition~\ref{prop:improvment-Ghosh-Puthen} considerably improves the results in \cite[Cor.~3.3 and Cor.~3.4]{GP19b}, all the while giving shorter, and more unified proofs. 

\begin{proposition}\label{prop:improvment-Ghosh-Puthen}
	Let $ M $ be an Ulrich $R$-module of dimension $s$. Let $L$ and $N$ be $R$-modules. Then the following hold true.
	\begin{enumerate}[\rm (1)]
		\item If there exists an integer $n\ge 1$ such that $ \Tor_i^R(M,N) = 0 $ for all $ n \le i \le n+s $, then $ \pd_R N < n+s $.
		\item If there exists an integer $n\ge 1$ such that $ \Ext^i_R(M,N)=0 $ for all $ n \le i \le n+s$, where $n+s\ge \depth_R N$, then $\id_R N<\infty$. 
		\item If there exists an integer $n\ge 1$ such that $ \Ext^i_R(L,M) = 0 $ for all $ n \le i \le n+s $, then $ \pd_R L < n $. 
	\end{enumerate} 
\end{proposition}

\begin{proof}
	Without loss of generality, passing to the faithfully flat extension $R[X]_{\fm[X]}$, we may assume that $k$ is infinite (\cite[Prop.~2.2.(3)]{GTT15}), and hence by \cite[Prop.~2.2.(2)]{GTT15}, there is a sequence $ {\bf x} = x_1, \dots, x_s $ in $ \fm $ such that $ \fm M = {\bf x} M $, i.e., $M/{\bf x} M$ is a nonzero $k$-vector space. When $s=0$, then $M$ is a $k$-vector space and (1) and (3) are immediate, and (2) follows from \cite[II. Thm.~2]{Rob76}. Now let $s\ge 1$. Since $\bf x$ is a system of parameters for $M$, hence $\bf x$ is $M$-regular since $M$ is CM \cite[2.1.2.(d)]{BH93}.   
	
	(1) Let $n \ge 1$ be an integer such that $\Tor_i^R(M,N)=0$ for all $ n \le i \le n+s$. In view of Lemma~\ref{lem:cons-vanishing-Tor}(1), $\Tor_i^R(M/{\bf x} M,N)=0$ for $i=n+s$. Since $M/{\bf x} M$ is a nonzero $k$-vector space, it follows that $\Tor_{n+s}^R(k,N)=0$, hence $\pd_R N < n+s$. 
	
	(2) Let $n\ge 1$ be an integer such that $\Ext^i_R(M,N)=0$ for all $n\le i\le n+s$. Similar to (1), we get (applying Lemma \ref{lem:cons-vanishing-Tor}(2)) that $\Ext^{n+s}_R(k,N)=0$. If $\id_R N=+\infty$, then by \cite[II. Thm.~2]{Rob76}, we know that $\Ext^i_R(k,N)\ne 0$ for all $i\ge \depth_R N$. Since $s+n\ge \depth_R N$, we must have $\id_R N<\infty$.   
	
	(3) If there exists an integer $n\ge 1$ such that $\Ext^i_R(L,M)=0$ for all $n\le i\le n+s$, then due to Lemma~\ref{lem:cons-vanishing-Tor}(3), we obtain that $\Ext^i_R(L,k)=0$ for $i=n+s-s=n$, hence $ \pd_R L < n $.  
\end{proof} 

\begin{remark} Note that when $M$ is an MCM Ulrich module, (i.e. $s=d$) the condition $n+s\ge \depth_R N$ of Proposition~\ref{prop:improvment-Ghosh-Puthen}(2) is automatically satisfied, and we directly recover \cite[Cor.~3.4]{GP19b}.  For some related results on Tor-rigid-test properties of MCM Ulrich modules over local Cohen--Macaulay rings of positive dimension, also see \cite[Theorem 5.10, 5.13]{dk}.    
\end{remark}

As an immediate consequence of Proposition~\ref{prop:improvment-Ghosh-Puthen}(1), we get

\begin{corollary}\label{cor:Ulrich-mod-d+1-Tor-rigid}
 	Let $ M $ be an Ulrich $R$-module of dimension $s$. Then $M$ is $ (s+1) $-$\Tor$-rigid-test.
\end{corollary} 

\begin{proof}
	For an $R$-module $N$, if $\Tor^R_i(M,N)=0$ for all $ n \le i \le n+s $ and for some $ n \ge 1 $, then $ \pd_R N < n+s $ by Proposition~\ref{prop:improvment-Ghosh-Puthen}(1), hence $\Tor^R_{i\ge n+s}(M,N)=0$.   
\end{proof}

Taking $n=1$ in Proposition~\ref{prop:improvment-Ghosh-Puthen}.(3), we get the following criteria for free modules.

\begin{corollary}
	Let $ M $ be an Ulrich $R$-module. For an $R$-module $L$, if $\Ext^i_R(L,M)=0$ for all $1\le i\le \dim (M)+1$, then $L$ is free.    
\end{corollary}

We note here some concrete class of Ulrich modules (of dimension $1$) over any CM local ring of positive dimension.

\begin{remark}\label{rmk:Ulrich-mod-of-dim-1}
	Suppose that $R$ is CM of dimension $d \ge 1$.  Let $x_1,\ldots,x_{d-1}$ be an $R$-regular sequence, and set $N:=R/(x_1,\ldots,x_{d-1})$. Then, for all $n\gg 0$, the $R$-modules $\fm^n N$ are Ulrich of dimension $1$.
\end{remark}

\begin{proof}
	Without loss of generality, we may assume that $k$ is infinite. Clearly $N$ is a CM $R$-module of dimension $1$.  By \cite[4.6.10]{BH93}, there exists $x\in \fm$ such that $(x)$ is a reduction of $\fm$ with respect to $N$. So $x\fm^n N=\fm^{n+1}N=\fm(\fm^n N)$ for all $n\gg 0$. Since $\depth \fm^n N>0$ and $\dim \fm^n N\le \dim N=1$, the $R$-module $\fm^n N$ is CM of dimension $1$. Thus $\fm^n N$ is Ulrich by \cite[Prop.~2.2.(2)]{GTT15}.
\end{proof}

For every reduced Gorenstein local ring of dimension $1$ and of minimal multiplicity which is not an integral domain, we show that there exist Ulrich modules that are neither rigid-test nor strongly rigid. In particular, it ensures that an Ulrich $R$-module need not be $d$-$\Tor$-rigid where $d \;(= \dim(R)) \ge 1$.

\begin{theorem}\label{thm:exam-reduced-Gor-dim-1-min-mult}
	Suppose that $R$ is a reduced Gorenstein local ring of dimension $1$ and of minimal multiplicity which is not an integral domain. Then $|\Min(R)| > 1$. Consider a non-empty proper subset $S \subsetneq \Min(R)$, and the ideals
	\begin{center}
		$I := \cap\{ \fp : \fp \in S \}$ \; and \; $J := \cap\{ \fq : \fq \in \Min(R) \smallsetminus S \}$.
	\end{center}
	Then the following statements hold true.
	\begin{enumerate}[\rm (1)]
		\item $R/I$ and $R/J$ are MCM Ulrich $R$-modules.
		\item For every positive integer $ n $, there exists an Ulrich module $ M_n $ satisfying
		\begin{center}
			$ \Tor^R_n(R/I, M_n) = 0 $ \; and \; $ \Tor^R_{n+1}(R/I, M_n) \neq 0 $.
		\end{center}
		In particular, $M_n$, $R/I$ and $R/J$ are neither Tor-rigid nor strongly-rigid.
	\end{enumerate}
\end{theorem}

\begin{example}
	The ring $R = k[[x,y]]/\langle x^2 - y^{2m} \rangle$ for every integer $ m \ge 1 $ satisfies the hypotheses of Theorem~\ref{thm:exam-reduced-Gor-dim-1-min-mult}, where $k$ is a field such that $\Char(k) \neq 2$. Note that $e(R) = 2 = \embdim(R) - \dim(R) + 1$, i.e., $R$ has minimal multiplicity.
\end{example}

In order to prove Theorem~\ref{thm:exam-reduced-Gor-dim-1-min-mult}, we need a few lemmas.

\begin{lemma}\label{1}
	Suppose that $\dim(R) \ge 1$. Let $ I \subsetneq \fm $ be a non-maximal radical ideal of $R$. Then, $ \depth(R/I) > 0 $.  
\end{lemma}

\begin{proof}
	Since $I$ is a non-maximal radical ideal, it follows that $ I = \bigcap_{\fp \in \Min(R/I)} \fp $ and $ \fm \notin \Min(R/I) $. Thus, $ \Ass(R/I) \subseteq \Min(R/I) $, and $ \fm \notin \Ass(R/I) $. Therefore $ \depth(R/I) \ne 0 $.   
\end{proof}

\begin{para}\label{para:reduced-nonzero-int}
	Let $S \subsetneq \Min(R)$ be a non-empty proper subset. If $R$ is a reduced local ring, then $\bigcap_{\fp \in S}\fp \ne 0$. Indeed, if $\bigcap_{\fp \in S}\fp = 0$, then for $ \fq \in \Min(R) \smallsetminus S $, we would have that $ \prod_{ \fp \in S } \fp \subseteq \bigcap_{ \fp \in S } \fp = 0 \subseteq \fq $, hence $ \fp \subseteq \fq $ for some $ \fp \in S $, which implies that $\fp=\fq\in \Min(R)\smallsetminus S$, a contradiction! This observation facilitates the next lemma.
\end{para}

\begin{lemma}\label{2}
	Suppose that $R$ is a reduced local ring which is not an integral domain. Then $|\Min(R)|>1$, and for any non-empty proper subset $ S \subsetneq \Min(R)$, denoting $ I := \bigcap_{ \fp \in S } \fp $ and $ J := \bigcap_{\fq \in \Min(R) \smallsetminus S} \fq $, it holds that $I$ and $J$ are nonzero proper ideals. Moreover, $R/I$ and $R/J$ have positive depth, and $I \cap J=0$.
\end{lemma}    

\begin{proof}
	Since $R$ is reduced, $ \bigcap_{\fp \in \Min(R)} \fp = 0 $, and hence $ |\Min(R)| > 1 $ as $R$ is not an integral domain.  Since $S$ and $ \Min(R) \smallsetminus S $ are non-empty proper subsets of  $\Min(R)$, it follows from \ref{para:reduced-nonzero-int} that $\bigcap_{\fp \in S} \fp \ne 0$ and $\bigcap_{\fq\in \Min(R)\smallsetminus S } \fq\ne 0$.  Since $\fm \notin \Min(R)$, the ideals $I$ and $J$ as defined in the statement are non-maximal radical ideals. So $R/I$ and $R/J$ have positive depth by Lemma~\ref{1}. Moreover, $ I \cap J = \bigcap_{\fp \in \Min(R)} \fp = 0 $.   
\end{proof}

\begin{proof}[Proof of Theorem~\ref{thm:exam-reduced-Gor-dim-1-min-mult}]
	(1) In view of Lemma~\ref{2}, the ideals $I$ and $J$ are nonzero proper ideals, and the $R$-modules $R/I$ and $R/J$ are MCM. Therefore, since $R$ is Gorenstein, both $R/I$ and $R/J$ are totally reflexive $R$-modules, hence are $l$-th syzygy modules for any $ l \ge 1 $. Since $R/I$ is annihilated by the nonzero ideal $I$, the $R$-module $R/I$ has no nonzero free summand. So, by \cite[Prop.~3.6]{umm}, $ R/I $ is an Ulrich $R$-module. For the same reason, $ R/J $ is also an Ulrich $R$-module.
	
	(2) By Lemma~\ref{2}, $ I \cap J = 0 $. So $ \Tor_1^R(R/I,R/J) \cong (I\cap J)/IJ=0$. Since $R/I$ and $R/J$ are non-free MCM modules, $\pd_R(R/I) = \pd_R(R/J) = \infty$. Therefore it follows from Proposition~\ref{prop:improvment-Ghosh-Puthen}.(1) that $ \Tor_2^R(R/I, R/J) \neq 0 $. From the discussion made in (1), since $R/J$ is totally reflexive, for every $n\ge 1$, there exists $R$-module $M_n$ such that $R/J\cong \syz^R_{n-1} M_n$. It follows that
	\begin{align*}
		& \Tor^R_n(R/I, M_n) \cong \Tor^R_1(R/I, \Omega_{n-1}^R(M_n)) = \Tor^R_1(R/I, R/J) = 0 \; \mbox{ and}\\
		& \Tor^R_{n+1}(R/I, M_n) \cong \Tor^R_2(R/I, \Omega_{n-1}^R(M_n)) = \Tor^R_2(R/I, R/J) \neq 0.
	\end{align*} 
	Note that $ \pd_R(M_n) = \infty $ as $ \pd_R(R/J) = \infty $. Hence $R/I$ and $M_n$ are neither $\Tor$-rigid nor strongly-rigid. For the same reason, $R/J$ is neither $\Tor$-rigid nor strongly-rigid. Since $R$ has minimal multiplicity, from the construction of $M_n$, one obtains that $M_n$ is Ulrich.
\end{proof}

The next proposition gives a class of Ulrich modules that are rigid-test over CM local rings of dimension $1$.

\begin{proposition}\label{prop:Ulrich-CM-dim-1-rigid-test}
	Let $R$ be a CM local ring of dimension $ 1 $, and $I$ be an ideal of positive height. If $I$ is an Ulrich $R$-module, then $I$ is rigid-test $($hence strongly rigid$)$.
\end{proposition}

\begin{proof}
	We show that if $n\ge 1$ is an integer, and $M$ is an $R$-module such that $\Tor^R_n(M,I)=0$, then $\pd_R M\le n$. Passing to the faithfully flat extension $R[X]_{\fm[X]}$, if necessary, we may assume that the residue field $k$ is infinite. Since $I$ is an Ulrich $R$-module, $\fm I=xI$ for some $R$-regular element $x\in \fm$. Therefore
	\begin{center}
		$ I \subseteq (\fm I:_R \fm)= (xI:_R \fm)\subseteq (xI:_R x) \subseteq I $,
	\end{center}
	where the last inclusion holds because $x$ is $R$-regular. Thus $I=(\fm I:_R \fm)$. Since $I$ has positive height, it is also $\fm$-primary. Note that $ \Tor^R_{n+1}(M, R/I) \cong \Tor^R_n(M, I) = 0 $. So it follows from \cite[Defn.~2.1, Thm.~2.10]{CGTT19} that $ \pd_R M \le n $, and we are done.    
\end{proof}

\begin{remark}\label{rmk:counterexample}
	Over a CM local ring of dimension $1$, a nonzero ideal $I$ of height $0$ which is MCM Ulrich as an $R$-module need not be rigid-test, see Example~\ref{exam:hypersurface-dim-1-Ext-Tor}.
\end{remark}

Finally, as consequences of Proposition~\ref{prop:improvment-Ghosh-Puthen}, we give some characterizations of regular and Gorenstein local rings in terms of Ulrich modules.

\begin{proposition}\label{prop:charac-regular-by-Ulrich-mod}
	Let $M$ and $N$ be Ulrich $R$-modules. $($Possibly, $M=N$$)$. Then
	\begin{enumerate}[\rm (1)]
		\item $R$ is regular $\Longleftrightarrow$ $\pd_R(M) < \infty$.
		\item $R$ is regular $\Longleftrightarrow$ $\id_R(N) < \infty$.
		\item Set $s:=\min\{\dim(M),\dim(N)\}$. The following statements are equivalent:
		\begin{enumerate}[\rm (a)]
			\item $R$ is regular.
			\item $\Ext_R^{n \le i \le n+s}(M,N)=0$ for some $n\ge 1$ with $n+\dim(M)\ge \depth_R N$.
			\item $\Tor^R_{n \le i \le n+s}(M,N)=0$ for some $n\ge 1$.
		\end{enumerate}
	\end{enumerate}
\end{proposition}

\begin{proof}
	(1) and (2): If $R$ is regular, then both $\pd_R(M)$ and $\id_R(N)$ are finite. For the reverse implications, let $\pd_R(M) < \infty$. Then $\Ext_R^{\gg 0}(M,k)=0$. It follows from Proposition~\ref{prop:improvment-Ghosh-Puthen}.(2) that $\id_R(k)$ is finite, hence $R$ is regular. For the second part, let $\id_R(N) < \infty$. Then $\Ext_R^{\gg 0}(k,N)=0$. By Proposition~\ref{prop:improvment-Ghosh-Puthen}.(3), this yields that $ \pd_R(k) $ is finite, hence $R$ is regular.
	
	(3) In view of Proposition~\ref{prop:improvment-Ghosh-Puthen}.(2) and (3), if $\Ext_R^{n \le i \le n+s}(M,N)=0$ for some $n\ge 1$, then $\pd_R (M) < \infty$ or $ \id_R(N) < \infty $ depending on whether $s=\dim(N)$ or $s=\dim(M)$ respectively. Hence, by (1) and (2), $R$ is regular. If $\Tor^R_{n \le i \le n+s}(M,N)=0$ for some $n\ge 1$, then by Proposition~\ref{prop:improvment-Ghosh-Puthen}.(1), $\pd_R (M) < \infty$ or $ \pd_R(N) < \infty $, hence (1) yields that $R$ is regular.
\end{proof}    

The example below shows that the number of consecutive vanishing of Ext or Tor in \ref{prop:charac-regular-by-Ulrich-mod}.(3) cannot be further improved (even for hypersurface of dimension $1$).

\begin{example}\label{exam:hypersurface-dim-1-Ext-Tor}
	Consider the local hypersurface $R = k[[x,y]]/(xy)$, where $k$ is a field, and $x,y$ are indeterminates.
	Set $M :=Rx$.
	Note that $M = R/\langle y \rangle$, which is an MCM $R$-module. Since $e(M) = 1 = \mu(M)$, the $R$-module $M$ is Ulrich. From the minimal free resolution
	\begin{equation}\label{reso-Rx}
	\cdots \stackrel{x\cdot}{\longrightarrow} R \stackrel{y\cdot}{\longrightarrow} R \stackrel{x\cdot}{\longrightarrow} R \stackrel{y\cdot}{\longrightarrow} R \longrightarrow 0
	\end{equation}
	of $M$, one computes that
	\begin{align*}
	&\Tor_{2n+1}^R(M,M) = (x)/(x^2) \neq 0 \mbox{ for all } n \ge 0,\\
	&\Tor_{2n}^R(M,M) = 0 \mbox{ for all } n \ge 1,\\
	&\Ext^{2n+1}_R(M,M) = 0 \mbox{ for all } n \ge 0, \mbox{ and
	}\\
	&\Ext^{2n}_R(M,M) = (x)/(x^2) \neq 0 \mbox{ for all } n \ge 1.
	\end{align*}
	Here $\dim(R)=1$, and $R$ is not regular.
	
	Note the ideal $I := Rx$ has height $0$, and it is MCM Ulrich as an $R$-module, but $I$ is not a rigid-test $R$-module. Thus Remark~\ref{rmk:counterexample} is verified.
\end{example}

Ulrich modules can also be used to characterize Gorenstein local rings. In \cite[Cor.~3.2]{GP19b}, it is shown that for an MCM Ulrich module $M$ over a CM local ring $R$, the base ring is Gorenstein \iff $\Ext_R^{n \le i \le n+d}(M,R)=0$ for some $n\ge 1$. More generally, we have (1) $\Leftrightarrow$ (3) in Proposition~\ref{prop:charac-Gor-by-Ulrich-mod} below.

\begin{proposition}\label{prop:charac-Gor-by-Ulrich-mod}
	Let $M$ be an Ulrich $R$-module of dimension $s$. Then the following are equivalent:
	\begin{enumerate}[\rm (1)]
		\item $R$ is Gorenstein.
		\item $\gdim_R(M) < \infty$.
		\item $\Ext_R^{n \le i \le n+s}(M,P)=0$ for some $n \ge 1$ with $n+s\ge \depth_R P$ and for some nonzero $R$-module $P$ with $\pd_R P < \infty$.  
		\item $\Ext_R^{n \le i \le n+s}(Q,M)=0$ for some $n \ge 1$ and for some nonzero $R$-module $Q$ with $\id_R Q < \infty$.
		\item $\Tor^R_{n \le i \le n+s}(M,Q)=0$ for some $n \ge 1$ and for some nonzero $R$-module $Q$ with $\id_R Q < \infty$.
	\end{enumerate}
\end{proposition}

\begin{proof}
	(1) $\Rightarrow$ (2): It is known for any $R$-module $M$, see, e.g., \cite[1.4.9]{Cr00}.
	
	(2) $\Rightarrow$ (3): Note that $\Ext_R^i(M,R)=0$ for all $i> \gdim_R(M)$, cf.~\cite[1.2.7]{Cr00}.
	
	(3) $\Rightarrow$ (1): Since $M$ is Ulrich, if $\Ext_R^{n \le i \le n+s}(M,P)=0$, then Proposition~\ref{prop:improvment-Ghosh-Puthen}.(2) yields that $\id_R P<\infty$. Thus $\id_R P<\infty$ and $\pd_R P < \infty$. Then, by a result of Foxby \cite[Cor.~4.4]{Fo77}, $R$ is Gorenstein.
	
	(1) $\Rightarrow$ (4) and (1) $\Rightarrow$ (5): These trivially follow by considering $Q=R$.
	
	(4) $\Rightarrow$ (1) and (5) $\Rightarrow$ (1): In view of Proposition~\ref{prop:improvment-Ghosh-Puthen}.(1) and (3), one obtains that $\pd_R Q < \infty$. Since $\id_R Q$ is also finite, by \cite[Cor.~4.4]{Fo77}, $R$ is Gorenstein.
\end{proof}  

\section{Complexity of Ulrich modules}\label{sec:Complexity of Ulrich modules}
	
	Here we study complexity of Ulrich modules, and obtain a few characterizations of various CM local rings via Ulrich modules. For every $ n \ge 0 $, let $ \beta_n^R(M) := \rank_k\left( \Tor^R_n(M,k) \right) $ denote the $n$th Betti number of $M$. The formal sum $ P_M^R(t) := \sum_{n \ge 0} \beta_n^R(M) t^n $ is called the Poincar\'{e} series of $ M $. The notions of complexity and curvature were introduced by Avramov.
	
	\begin{definition}\label{def:complexity}
		(1) The complexity of $ M $, denoted $ \cx_R(M) $, is the smallest non-negative integer $ b $ such that $ \beta_n^R(M) \le \alpha n^{b-1} $ for all $ n \gg 0 $, and for some real number $ \alpha > 0 $. If no such $ b $ exists, then $ \cx_R(M) := \infty $.
		
		(2) The curvature of $ M $, denoted $ \curv_R(M) $, is the reciprocal value of the radius of convergence of $ P_M^R(t) $, i.e.,
		\[
		\curv_R(M) := \limsup_{n \to \infty} \sqrt[n]{\beta_n^R(M)}.
		\]
	\end{definition}
	
	We use the following well-known properties of complexity and curvature.
	
	\begin{lemma}\label{lem:cx-properties}
		The following statements hold true.
		\begin{enumerate}[\rm (i)]
			\item
			{\rm (a)} $ \pd_R(M) < \infty \; \Longleftrightarrow \; \cx_R(M) = 0 \; \Longleftrightarrow \; \curv_R(M) = 0 $.\\
			{\rm (b)} $ \cx_R(M) < \infty \; \Longrightarrow \; \curv_R(M) \le 1 $.
			\item
			{\rm (a)} $ \cx_R(M) = \cx_R\left( \Omega_n^R(M) \right) $ and $ \curv_R(M) = \curv_R\left( \Omega_n^R(M) \right) $ for all $ n \ge 1 $.\\
			{\rm (b)} $ \cx_R(M\oplus N) = \max\{ \cx_R(M), \cx_R(N) \} $.\\
			{\rm (c)} $ \curv_R(M\oplus N) = \max\{ \curv_R(M), \curv_R(N) \} $.	
			\item
			Let $ x \in R $ be regular on both $ R $ and $ M $. Then 
			\[
				\cx_{R/(x)}(M/xM) = \cx_R(M) \quad \mbox{and} \quad \curv_{R/(x)}(M/xM) = \curv_R(M).
			\]
		\end{enumerate}
	\end{lemma}
	
	\begin{proof}
		(i) and (ii) can be seen in \cite[4.2.3 and 4.2.4]{Avr98}.
		
		(iii) Since $ x $ is regular on both $ R $ and $ M $, for every $ n \ge 0 $, the $n$th Betti number of $M/xM$ over $R/xR$ is same as that of $M$ over $R$, see, e.g., \cite[pp.~140, Lem.~2]{Mat86}. Hence the desired equalities follow.
	\end{proof}  

	The following lemma could also be deduced from \cite[3.3.5.(1)]{Avr98}, however our proof is more elementary, and it exactly shows the part where $x\notin \fm^2$ is needed.
	
	\begin{lemma}\label{lem:cx-mod-x}
	Let $x\in \fm\smallsetminus \fm^2$ be an $R$-regular element. Let $M$ be an $R$-module such that $xM=0$. Then $\cx_{R/(x)}(M)=\cx_R(M)$ and $\curv_{R/(x)}(M)=\curv_R (M)$.
%
	\end{lemma} 
	\begin{proof}
		Let $R'=R/(x)$. For every $n\ge 2$, note that
		\begin{equation}\label{betti-tor}
		\beta_n^R(M) = \rank_{k} \Tor^R_{n-1}(\fm,M) = \rank_k\Tor^{R'}_{n-1}(\fm/x\fm,M),
		\end{equation}
		where the last equality can be observed from \cite[pp.~140, Lem.~2]{Mat86}. Since $x\in \fm\smallsetminus \fm^2$, it follows that $k$ is a direct summand of $\fm/x\fm$ (see, e.g., \cite[Cor.~5.3]{Tak06}). Hence \eqref{betti-tor} yields that $\beta_n^R(M)\ge \beta_{n-1}^{R'}(M)$ for all $n\ge 2$. So $\cx_R(M)\ge \cx_{R'}(M)$ and $\curv_R(M)\ge \curv_{R'}(M)$. Also from the long exact sequence of Tor as described in \cite[Lem.~2.7]{GP19b}, one obtains that $\beta_n^R(M)\le \beta_n^{R'}(M)+\beta_{n-1}^{R'}(M)$, hence $\cx_R(M)\le \cx_{R'}(M)$ and $\curv_{R}(M)\le \curv_{R'}(M)$. Thus the desired equalities hold true.
	\end{proof}

\begin{remark}
	Regarding inequalities of complexities and curvatures of modules along deformation of rings, we point out that \cite[4.2.5.(3)]{Avr98} does not hold true always. For example, consider $R = k[[T]]$, $R' = R/(T^2)$ and $M=k$. In this setup, $\cx_{R'}(k) = 1 \nleqslant 0 = \cx_R(k)$ and $\curv_{R'}(k) = 1 \nleqslant 0 = \curv_R(k)$.
\end{remark}
	
	Let us recall a few well-known properties of complexity of the residue field.
	
	\begin{proposition}{~}\label{prop:facts-projcx-k}
		\begin{enumerate}[\rm (i)]
			\item
			{\rm \cite[Prop.~2]{Avr96}} The residue field has maximal complexity and curvature, i.e.,
			\begin{align*}
				\cx_R(k) & = \sup\{ \cx_R(M) : M \mbox{ is an $ R $-module} \} \,\mbox{ and}\\
				\curv_R(k) & = \sup\{ \curv_R(M) : M \mbox{ is an $ R $-module} \}.
			\end{align*}
			\item 
			$R \mbox{ is complete intersection } \Longleftrightarrow \cx_R(k) < \infty \Longleftrightarrow \curv_R(k) \le 1$, see, e.g., \cite[8.1.2 and 8.2.2]{Avr98}
		\end{enumerate}
	\end{proposition}
	
	\begin{definition}\label{defn:max-complexity-curvature}
		An $ R $-module $ M $ is said to have maximal complexity (resp., curvature) if $ \cx_R(M) = \cx_R(k) $ (resp., $ \curv_R(M) = \curv_R(k) $).
	\end{definition}

	\begin{theorem}\label{thm:Ulrich-max-cx}
		Every Ulrich module over a CM local ring has maximal complexity and curvature.
	\end{theorem}
	
	\begin{proof}
		Suppose that $R$ is CM, and $M$ is an Ulrich $R$-module. We prove that $\cx_R(M) = \cx_R(k)$ and $\curv_R(M) = \curv_R(k)$ by using induction on $s:=\dim(M)$. We may assume that the residue field $k$ is infinite. In the base case, assume that $s = 0$. Since $M$ is Ulrich, in this case $\fm M = 0$, i.e., $M$ is a nonzero $k$-vector space, hence the desired equalities are obvious.
		
		Next assume that $s \ge 1$. There is an $R\oplus M$-superficial element $x$. Since $\depth(R\oplus M) = s \ge 1$, the element $x \in \fm \smallsetminus \fm^2 $, and it must be regular on $R\oplus M$. Setting $\overline{(-)} := (-) \otimes_R R/xR$, the module $\overline{M}$ is Ulrich over the CM local ring $\overline{R}$ of dimension $s-1$. So, by the induction hypothesis,
		\begin{equation}\label{cx-curv-M/xM-k}
		\cx_{\overline{R}}(\overline{M}) = \cx_{\overline{R}}(k) \quad \mbox{ and } \quad \curv_{\overline{R}}(\overline{M}) = \curv_{\overline{R}}(k).
		\end{equation}
		In view of Lemma~\ref{lem:cx-properties}.(iii), we have
		\begin{equation}\label{cx-curv-M/xM-M}
			\cx_{\overline{R}}(\overline{M}) = \cx_R(M) \quad \mbox{ and } \quad \curv_{\overline{R}}(\overline{M}) = \curv_R(M).
		\end{equation}
		Since $x \in \fm \smallsetminus \fm^2 $ is $R$-regular, by Lemma~\ref{lem:cx-mod-x}, $\cx_{\overline{R}}(k) = \cx_R(k)$ and $\curv_{\overline{R}}(k) = \curv_R(k)$. Therefore it follows from \eqref{cx-curv-M/xM-k} and \eqref{cx-curv-M/xM-M} that $\cx_R(M) = \cx_R(k)$ and $\curv_R(M) = \curv_R(k)$. This completes the proof.
	\end{proof}

	As a consequence of Theorem~\ref{thm:Ulrich-max-cx}, we obtain the following new characterizations of complete intersection local rings in terms of Ulrich modules.
	
	\begin{corollary}\label{cor:charac-CI-via-Ulrich}
		Suppose that $R$ is CM. Let $M$ be an Ulrich $R$-module. Then the following are equivalent:
		\begin{enumerate}[\rm (1)]
			\item $R$ is complete intersection $($of codimension $c$$)$.
			\item $\cid_R(M) < \infty$.
			\item $\cx_R(M) < \infty$ $($and $\cx_R(M) = c$$)$.
			\item $\curv_R(M) \le 1$.
		\end{enumerate}
	\end{corollary}
	
	\begin{proof}
		The implications (1) $\Rightarrow$ (2) $\Rightarrow$ (3) hold true for any $R$-module \cite[(1.3), (5.6)]{AGP97}. The other implications are obtained from Theorem~\ref{thm:Ulrich-max-cx} and Proposition~\ref{prop:facts-projcx-k}.(ii).
	\end{proof}
	
	In the following corollary, by an abstract hypersurface, we mean a complete intersection local ring of codimension at most $1$. It is well known that every CM local ring of codimension at most $1$ is an abstract hypersurface. The result stated in the corollary below can also be deduced from \cite[5.2.8 and 5.3.3.(2)]{Avr98}. However our proof is more direct.
	
	\begin{corollary}\label{cor:min-mult-comlete-inter}
		Suppose that $R$ is CM of minimal multiplicity, and it admits a module $ M $ such that $ 0 < \cx_R(M) < \infty $ $($or $ 0 < \curv_R(M) \le 1 $$)$. Then $ R $ is a complete intersection ring of codimension $1$, i.e., $R$ is an abstract hypersurface.
	\end{corollary}  
	
	\begin{proof} 
		As Betti numbers of a module and (co)dimension of a ring remain same after passing through the faithfully flat extension $R[X]_{\fm[X]}$, so are complexity, projective dimension and complete intersection properties (cf.~\cite[2.3.3.(b)]{BH93}). Thus we may assume that the residue field $k$ is infinite. From the assumption on $M$, it follows that $ \pd_R (M) = \infty $. Hence, $R$ is singular, and $\Omega_{d+1}^R (M) $ is an Ulrich $R$-module \cite[Prop. 3.6]{umm}. Moreover, by the assumption, $ \cx_R(\Omega_{d+1}^R(M)) = \cx_R(M) < \infty$ or $\curv_R(\Omega_{d+1}^R(M)) = \curv_R(M) \le 1$. So, by Corollary~\ref{cor:charac-CI-via-Ulrich}, $R$ is complete intersection (hence Gorenstein). Since $R$ has minimal multiplicity, there exists an $R$-regular sequence $x_1,\dots,x_d \in \fm$ such that $\fm^2 = (x_1,\dots,x_d) \fm$. Put $ \overline{(-)} := (-) \otimes_R R/(x_1,\dots,x_d) $. Then $ \overline{\fm}^2 = 0 $, so $\overline{\fm } \subseteq (0 : \overline{\fm}) = \Soc(\overline{R})$. Since $\overline R$ is also Gorenstein, $\dim_k(\Soc(\overline{R}))=1$, hence $ \dim_k (\overline{\fm}) \le 1 $. But $ R $ is singular, so $ \overline{\fm} \ne 0 $, thus $ \dim_k (\overline{\fm}) = 1 $. Therefore $ \fm/(x_1,\dots,x_d) = y \overline R = (y,x_1,\dots,x_d)/(x_1,\dots,x_d)$ for some $y\in R$, hence $\fm=(y,x_1,\dots,x_d)$. Thus $R$ has codimension $1$. 
	\end{proof}  

	Next we verify Theorem~\ref{thm:Ulrich-max-cx}, Corollaries~\ref{cor:charac-CI-via-Ulrich} and \ref{cor:min-mult-comlete-inter} for a class of CM local rings by explicitly computing complexity and curvature of some Ulrich modules.
	
	\begin{example}\label{exam:complexity-Ulrich}
		Consider the $1$-dimensional CM local ring
		\begin{center}
			$R = k[[x_1,\ldots,x_n]]/\langle x_i x_j : 1 \le i < j \le n \rangle$,
		\end{center}
		where $k$ is a field, and $x_1,\ldots,x_n$ are indeterminates. Since
		\begin{center}
			$e(Rx_1) = e(R/\langle x_2,\ldots,x_n \rangle) = e(k[[x_1]]) = 1 = \mu(Rx_1)$,
		\end{center}
		the $R$-module $Rx_1$ is Ulrich. Note that $Rx_1 \cong R/\langle x_2,\ldots,x_n \rangle$, which is an MCM $R$-module. Thus, for every $1 \le j \le n$, the $R$-module $Rx_j$ is MCM and Ulrich. As $\Omega_1^R(k) = \fm = \langle x_1,\ldots,x_n \rangle = Rx_1 \oplus R x_2 \oplus \cdots \oplus R x_n$, it follows that
		\begin{align*}
		&\cx_R(k) = \cx_R(\fm) = \max\{ \cx_R(Rx_i) : 1 \le i \le n \} = \cx_R(Rx_j) \;\mbox{ and}\\
		&\curv_R(k) = \max\{ \curv_R(Rx_i) : 1 \le i \le n \} = \curv_R(Rx_j)
		\end{align*}
		for every $1 \le j \le n$. To compute $\cx_R(Rx_j)$ and $\curv_R(Rx_j)$ explicitly, note that
		\[
			\Omega_m^R(Rx_1) = \Omega_{m-1}^R(\langle x_2,\ldots,x_n \rangle) = \Omega_{m-1}^R(R x_2) \oplus \cdots \oplus \Omega_{m-1}^R(R x_n)
		\]
		for all $m\ge 1$, and $\beta_m^R(Rx_1) = \mu(\Omega_m^R(Rx_1))$. An induction on $m$ yields that
		\begin{equation}\label{betti-numbers}
		\beta_m^R(Rx_1) = (n-1)^m \quad \mbox{for all } m \ge 1.
		\end{equation}
		Computing complexity and curvature of $Rx_1$ from \eqref{betti-numbers}, one obtains that
		\[
			\curv_R(Rx_1) = n-1
			\quad \;\mbox{ and} \quad
			\cx_R(Rx_1) = \left\{ \begin{array}{ll}
			0 & \mbox{if } n = 1,\\
			1 & \mbox{if } n = 2,\\
			\infty & \mbox{if } n \ge 3.
			\end{array} \right.			
		\]
		Note that $R$ is complete intersection when $n\le 2$. In the case of $n\ge 3$, since $x_1 + \cdots + x_n$ is $R$-regular, and the quotient ring
		\begin{center}
			$R/\langle x_1 + \cdots + x_n \rangle \cong k[[x_2,\ldots,x_n]]/\langle x_i x_j : 2 \le i \le j \le n \rangle$
		\end{center}
		is of type $n-1 \ge 2$, it follows that $R/\langle x_1 + \cdots + x_n \rangle$ is not Gorenstein. Thus $R$ is complete intersection \iff $n\le 2$. Note that $e(R) = n = \edim(R) - \dim(R) + 1$, hence $R$ has minimal multiplicity.  
	\end{example}

	Corollary~\ref{cor:min-mult-comlete-inter} does not hold true for arbitrary CM local rings.
	
	\begin{example}
		The ring $R = k[[x_1,\ldots,x_n,y]]/\left( \langle x_1,\ldots,x_n \rangle^2 + \langle y^2 \rangle \right)$ for every $n\ge 1$ is an Artinian local ring. It does not have minimal multiplicity. Set $M := R/(y)$. From the minimal free resolution $ \cdots \to R \stackrel{y}{\to} R \stackrel{y}{\to} R \to 0$ of $M$, one computes that $\cx_R(M) = 1$ and $\curv_R(M) = 1$. Note that $R$ is complete intersection if and only if $n = 1$. In fact, if $ n \ge 2 $, then $ \Soc(R) = \langle x_1 y, \ldots, x_ny \rangle $ is not cyclic, hence $ R $ is not Gorenstein.
	\end{example} 
\section{Ulrich modules over CM rings of minimal multiplicity}

%

The existence of an MCM Ulrich module over an arbitrary CM local ring is still an open question \cite[Sec.~3]{Ulr84}. However, this question has affirmative answers for some particular classes of rings. For example, if $R$ is CM of dimension $1$, then $\fm^{e(R)-1}$ is an MCM Ulrich $R$-module \cite[(2.1)]{BHU87}. See also \cite[Proof of Cor.~3.2]{Ulr84}. In Remark~\ref{rmk:Ulrich-mod-of-dim-1}, we saw that over any CM local ring of positive dimension, there always exist Ulrich modules of dimension $1$.

If $R$ is CM of minimal multiplicity, then for every $n\ge d$, $\Omega_n^R(k)$ is Ulrich \cite[(2.5)]{BHU87}. We considerably strengthen this fact.

\begin{proposition}\label{prop:exam-Ulrich-hom-image-syz-k}
	Suppose $R$ is CM of minimal multiplicity. Let $M$ be an MCM homomorphic image of a finite direct sum of syzygy modules of $k$. Then $M$ is Ulrich.
\end{proposition}

\begin{proof}
	We can pass to the faithfully flat extension $R[X]_{\fm[X]}$ and assume that the residue field $k$ is infinite (\cite[Prop.~2.2.(3)]{GTT15}). If $R$ is a field, then $M$ is an $R$-vector space, which is Ulrich. So we may assume that $R$ is not a field. Since $k$ is infinite, we can choose a minimal reduction $x_1,\cdots,x_d\in \fm \smallsetminus \fm^2$ of $\fm$, hence $\fm^2=(x_1,\cdots,x_d)\fm$, and $x_1,...,x_d$ is an $R$-regular sequence. Set $\overline R:=R/(x_1,\dots,x_d)$, and for an $R$-module $N$, let $\overline N$ denote the $\overline R$-module $ N \otimes_R \overline R \cong N/(x_1,\dots,x_d)N $. Note that $ \overline R $ is Artinian with the maximal ideal $\overline{\fm}$ such that $\overline{\fm}^2=0$. So $ \overline \fm \subseteq \Soc(\overline R) $. Since $M$ is MCM, $ x_1,\dots,x_d $ is also $ M $-regular. Therefore, along with \cite[Cor.~5.3]{Tak06}, we get that $ \overline{M} $ is a homomorphic image of a finite direct sum of some $ \overline R $-syzygy modules of $k$. So, by \cite[Lem.~2.1]{GGP18}, it follows that $\overline \fm \subseteq \Soc(\overline R) \subseteq \ann_{\overline R}(\overline M)$. Hence $ M/(x_1,\dots,x_d)M $ is annihilated by $\fm$ as an $R$-module, i.e., $\fm M=(x_1,\dots,x_d)M $. Thus $ M $ is an Ulrich $R$-module.  
\end{proof}

Let $\Mod(R)$ denote the category of all (finitely generated) $R$-modules.
	
\begin{lemma}\label{funreg}
	Let $N$ be an $R$-module and $x$ be an $N$-regular element. Let $F$ be a half-exact covariant linear functor on $\Mod(R)$. Then, there is an embedding $ F(N)/xF(N) \hookrightarrow F(N/xN) $.   
\end{lemma}

\begin{proof}
	Since $F$ is half-exact, covariant and linear, the exact sequence $ 0 \to N \xrightarrow{ x\cdot } N \xrightarrow{\pi} N/xN \to 0 $ induces another exact sequence $F(N) \xrightarrow{x\cdot} F(N) \xrightarrow{F(\pi)} F(N/xN)$, which gives $\dfrac{F(N)}{xF(N)}=\dfrac{F(N)}{\ker F(\pi)} \cong \text{Im} F(\pi) \hookrightarrow F(N/xN)$. 
\end{proof}

\begin{remark}
	For every fixed integer $i$ and each $R$-module $X$, the functors
	$\Tor^R_i(X,-)$ and $\Ext^i_R(X,-)$ are covariant half-exact linear functors.  
\end{remark} 


\begin{corollary}\label{funul}
	Assume that $R$ is CM of minimal multiplicity, infinite residue field and dimension $1$. Let $\{F_i\}_{i=0}^l$ be a finite collection of half-exact covariant linear functors on $\Mod(R)$, and let $M$ be an MCM $R$-module which is a homomorphic image of a finite direct sum of copies of $F_i(\syz^R_j(k))$. Then, $M$ is Ulrich.
\end{corollary}   

\begin{proof}
	We may assume that $R$ is not a field. There exist integers $n_{ij}\ge 0$ such that $\bigoplus_{i,j} F_i(\syz^R_j(k))^{\oplus n_{ij}} \to M \to 0$. Choose $x\in \fm\smallsetminus \fm^2$ with $\fm^2=x\fm$ so that $R/xR$ is Artinian and has minimal multiplicity. Since $x$ is $R$-regular, it is also $M$-regular. We have a surjection
	\begin{center}
		$\bigoplus_{i,j} \left(F_i(\syz^R_j(k))/xF_i(\syz^R_j(k))\right)^{\oplus n_{ij}} \longrightarrow M/xM \to 0$.
	\end{center}
	For $j=0$, since $k$, and hence $F_i(k)$ is already killed by $\fm$, so $F_i(k)/xF_i(k)=F_i(k)$ is killed by $\overline \fm$ the maximal ideal of $R/xR$. For $j>0$, $x$ is $\syz^R_j(k)$-regular (since it is $R$-regular), so by Lemma~\ref{funreg} $F_i(\syz^R_j(k))/xF_i(\syz^R_j(k))$ is a submodule of $F_i\left(\syz^R_j(k)/x\syz^R_j(k)\right)\cong F_i\left(\syz_{j-1}^{R/xR}(k)\oplus \syz_j^{R/xR}(k)\right)$, see \cite[Cor.~5.3]{Tak06}. By \cite[Lem.~2.1]{GGP18}, each $ \syz_{j-1}^{R/xR}(k) \oplus \syz_j^{R/xR}(k) $ is killed by $\Soc(R/xR)$, which contains $\overline \fm$ since $R/xR$ has minimal multiplicity. Thus, due to linearity of each $F_i$, each $F_i \left(\syz^R_j(k)/x\syz^R_j(k)\right) \cong F_i \left( \syz_{j-1}^{R/xR}(k) \oplus \syz_j^{R/xR}(k) \right)$ is killed by $\overline \fm$, hence its submodule $F_i(\syz^R_j(k))/xF_i(\syz^R_j(k))$ is also killed by $\overline \fm$. It follows that $M/xM$ is killed by $\overline \fm$ as $R/xR$-module. So $\fm M=xM$, hence $M$ is Ulrich.   
\end{proof}

\begin{remark}
	Since $\Hom_R(R,-)$ is a half-exact covariant linear functor, so Corollary~\ref{funul} generalizes Proposition~\ref{prop:exam-Ulrich-hom-image-syz-k} in dimension $1$ case.  
\end{remark}

For the contravariant functor case, one needs cohomological $\delta$-functors

\begin{lemma}\label{delta}
	Let $T=(T^i,\delta^i)$ be a linear contravariant cohomological $\delta$-functor on $\Mod(R)$. Let $N$ be an $R$-module and $x$ be an $N$-regular element. Then, for each $i$, there is an embedding $\dfrac{T^i(N)}{xT^i(N)}\hookrightarrow T^{i+1}(N/xN)$.  
\end{lemma}

\begin{proof}
	The short exact sequence $0\to N \xrightarrow{x\cdot} N \xrightarrow{\pi} N/xN \to 0$ induces an exact sequence $T^i(N/xN)\xrightarrow{T^i(\pi)}T^i(N)\xrightarrow{x\cdot} T^i(N)\xrightarrow{\delta^i} T^{i+1}(N/xN)$ for each $i$. Hence, $\dfrac{T^i(N)}{xT^i(N)}=\dfrac{T^i(N)}{\ker \delta^i}\cong \text{Im} \delta^i\hookrightarrow T^{i+1}(N/xN)$.   
\end{proof}

\begin{remark}
	For every fixed $R$-module $X$, the sequence $T(-):=\left(\Ext^i_R(-,X)\right)_{i\ge 0}$ gives a linear contravariant cohomological $\delta$-functor.  
\end{remark}  

\begin{corollary}\label{funul:2nd}
	Assume that $R$ is CM of minimal multiplicity, infinite residue field and dimension $1$. Let $\{T_j\}_{j=0}^l$ be a finite collection of cohomological contravariant linear $\delta$-functors, and write $T_j=(T_j^i,\delta^i_j)$ for each $j$. Suppose there exist non-negative integers $i_j,n_{ju}$ such that $\bigoplus_{j,u} T_j^{i_j}(\syz_u^R(k))^{\oplus n_{ju}}$ has an MCM homomorphic image $M$. Then, $M$ is Ulrich. 
\end{corollary}

\begin{proof}
	The proof is similar to that of Corollary~\ref{funul} by using Lemma~\ref{delta} in place of Lemma~\ref{funreg}.   
\end{proof}

%

As a consequence of Proposition~\ref{prop:exam-Ulrich-hom-image-syz-k}, we improve \cite[Thm.~4.1]{Gho} and \cite[Prop.~5.2]{GP19b}.

\begin{corollary}
	Suppose that $R$ is CM of minimal multiplicity. Assume $M$ and $N$ are nonzero $R$-modules which are homomorphic images of finite direct sums of some syzygy modules of $k$. Assume that $M$ is MCM. Then the following are equivalent:
	\begin{enumerate}[\rm (1)]
		\item $R$ is regular.
		\item $\Tor^R_i(M,N)=0$ for some $(d+1)$  consecutive values of $i\ge 1$.
		\item $\Ext^i_R(N,M)=0$ for some $(d+1)$  consecutive values of $i\ge 1$.
	\end{enumerate}
\end{corollary}

\begin{proof}
	By Proposition~\ref{prop:exam-Ulrich-hom-image-syz-k}, $M$ is Ulrich. In either of the cases (2) and (3), by Proposition~\ref{prop:improvment-Ghosh-Puthen}(1) and (3), we get that $\pd_R N<\infty$. Hence it follows from \cite[Prop.~7]{Mar96} that $R$ is regular.
\end{proof}

The following gives a refinement of the vanishing threshold of \cite[Thm.~4.3]{GP19b}.

\begin{proposition}\label{prop:vanishing-Tor-and-finite-pd}
	Suppose that $R$ is CM of minimal multiplicity. Let $M$ and $N$ be $R$-modules. Let $ t := \depth_R M $. If $ \Tor^R_i(M,N) = 0 $ for some $ (d+1) $  consecutive values of $ i \ge d-t+2 $, then either $ \pd_R M < \infty $ or $ \pd_R N < \infty $, and we also have $\Tor^R_{i\ge d-t+1}(M,N)=0$. 
\end{proposition}   

\begin{proof}
	Assume that $ \pd_R M = \infty $. Since $R$ has minimal multiplicity, the syzygy module $ \Omega_{d-t+1}^R(M) $ is MCM and Ulrich \cite[Prop. 3.6]{umm}. From the hypothesis, it follows that
	\begin{center}
		$ \Tor^R_i(\Omega_{d-t+1}^R(M), N) = 0 $ for $ (d+1) $  consecutive values of $ i \ge 1 $,
	\end{center}
	and then Proposition~\ref{prop:improvment-Ghosh-Puthen}(1) implies that $\pd_R N<\infty$.
	
	 Finally, $\Tor^R_{i\ge d-t+1}(M,N)=\Tor^R_{i\ge 1}(\Omega^R_{d-t}(M),N) = 0$ if $\pd_R M<\infty $ (as then $\Omega^R_{d-t}(M)$ is free), or also if $\pd_R N <\infty$ (by \cite[Lem.~2.2]{Yo}).  
\end{proof}

\begin{corollary}\label{minrigid}
	Suppose that $R$ is CM of minimal multiplicity. Let $M$ be an $R$-module, and $ t := \depth_R M $. Then $ M $ is $ (2d-t+2) $-$ \Tor $-rigid. In particular, every MCM $ R $-module is $ (d+2) $-$ \Tor $-rigid, and every module is $ (2d+2) $-$ \Tor $-rigid.
\end{corollary}

\begin{proof}
	Let $N$ be an $R$-module, and suppose there exists an integer $ n \ge 1 $ such that $ \Tor^R_{n \le i \le n+2d-t+1}(M,N) = 0 $. Then $ \Tor^R_{n+d-t+1 \le i \le n+2d-t+1}(M,N) = 0 $, and these are $ (n+2d-t+1)-(n+d-t+1)+1 = d+1 $  consecutive values $ \ge n+d-t+1 \ge d-t+2 $. Hence, by Proposition~\ref{prop:vanishing-Tor-and-finite-pd}, either $ \pd_R M < \infty $ or $ \pd_R N < \infty $, i.e., either $ \pd_R M \le d $ or $ \pd_R N \le d $. Since $ n+2d-t+1 > d+1 $, it follows that $ \Tor^R_{ i \ge n+2d-t+1 }(M,N) = 0 $. The claim about MCM module follows since then $ t = d $. The claim about all modules follows since $ 2d+2 \ge 2d-t+2 $, and if a module is $j$-$\Tor$-rigid, then obviously it is $(j+1)$-$\Tor$-rigid.
\end{proof}

Next we give a criterion for a module over a CM local ring of minimal multiplicity to be free in terms of vanishing of certain Ext modules. For that, we need two elementary lemmas, which are possibly well known.

\begin{lemma}\label{nonvanish}
	Let $M \neq 0$ be an $R$-module such that $ \pd_R M = n < \infty $. Then $ \Ext^n_R(M,R) \ne 0 $. 
\end{lemma}

\begin{proof}
	Choose a minimal free resolution $0\to F_n\xrightarrow{\partial}F_{n-1}\to \cdots \to F_0 \to 0$ of $M$. Then, $\Ext^n_R(M,R)=F_n^*/\text{Im} (\partial^*)$, where $(-)^* := \Hom_R(-,R)$. Since the entries of $\partial$ are in $\fm$, the entries of $\partial^*:F_{n-1}^*\to F_n^*$ are also in $\fm$, hence $F_n^*/\text{Im} (\partial^*) \neq 0$.
\end{proof}

\begin{lemma}\label{lem:Ext-vanishing}
	Let $l,m \ge 0$ and $n \ge 1$ be integers. Let $\Ext^i_R(M,N)=0$ for all $n\le i 
	\le n+m$, and $\Ext^j_R(M, R)=0$ for all $n+1\le j \le n+l+m$. Then
	\begin{enumerate}[\rm (1)]
		\item $\Ext^j_R(M, \Omega_l^R(N))=0 \mbox{ for all } n+l\le j \le n+l+m$.
		\item $\Ext^j_R(\Omega_l^R(M), \Omega_l^R(N))=0 \mbox{ for all } n\le j \le n+m$.
	\end{enumerate}
\end{lemma}

\begin{proof}
	(1) We use induction on $l$. There is nothing to prove when $l=0$. Suppose $l \ge 1$. Consider an exact sequence $0\to \syz_l^R(N) \to F \to \syz_{l-1}^R(N) \to 0$, where $F$ is a free $R$-module. It induces another exact sequence
	\begin{equation}\label{s.e.s-Ext}
		\Ext^i_R(M,\syz_{l-1}^R(N))\to \Ext^{i+1}_R(M,\syz_l^R(N))\to \Ext^{i+1}_R(M,F)
	\end{equation}
	for each $i$. By the induction hypothesis, we have
	\begin{equation}\label{Ext-vanshing}
		\Ext^i_R(M, \Omega_{l-1}^R(N))=0 \mbox{ for all } n+(l-1)\le i \le n+(l-1)+m.
	\end{equation}
	Along with the given hypothesis, it follows from \eqref{s.e.s-Ext} and \eqref{Ext-vanshing} that \begin{center}
		$\Ext^j_R(M,\syz_l^R(N)) = 0$ for all $n+l\le j \le n+l+m$.
	\end{center}
	This completes the inductive step, and hence the proof of (1).
	
	(2) For every $ n\le j \le n+m$,
	\begin{align*}
		\Ext^j_R(\Omega_l^R(M), \Omega_l^R(N)) &\cong \Ext^{j+1}_R(\Omega_{l-1}^R(M), \Omega_l^R(N)) \cong \cdots \\ &\cong \Ext^{j+l}_R(M, \Omega_l^R(N)) = 0 \; \mbox{ [by (1)]}
	\end{align*}
	as $n+l\le j+l \le n+l+m$.  
\end{proof}

\begin{proposition}\label{prop:min-mult-ARC}
	Suppose $R$ is CM of minimal multiplicity. Let $n\ge 1$ be an integer. Set $t := \depth(N)$. Let $\Ext^i_R(M,N)=\Ext^j_R(M, R)=0$ for all $n\le i \le n+d$ and $n\le j \le n+2d-t+1$. Then, either $\pd_R M \le n-1$, or $\pd_R N < \infty$.
\end{proposition} 

\begin{proof}
	Suppose $\pd_R N = \infty$. Note that $\syz_{d-t}^R(N)$ is MCM. Since $R$ has minimal multiplicity, $\syz_{d-t+1}^R(N)$ is  Ulrich. In view of the given hypothesis, by Lemma~\ref{lem:Ext-vanishing}, it follows that $\Ext^j_R(M,\syz_{d-t+1}^R(N)) = 0$ for all $n+d-t+1\le j \le n+2d-t+1$.
	Therefore, by Proposition~\ref{prop:improvment-Ghosh-Puthen}.(3), $\pd_R M \le n+d-t$. Moreover, since $\Ext^j_R(M, R)=0$ for all $n\le j \le n+2d-t+1$, Lemma~\ref{nonvanish} yields that $\pd_R M \le n-1$.
\end{proof}

	As a consequence of Proposition~\ref{prop:min-mult-ARC}, we have the following freeness criteria.
	
\begin{corollary}\label{cor:ARC-min-mult}
	Suppose $R$ is CM of minimal multiplicity. Set $t = \depth(M)$. Let
	\begin{center}
		$\Ext^i_R(M,M)=\Ext^j_R(M, R)=0$ for all $1\le i \le d+1$ and $1\le j \le 2d-t+2$.
	\end{center}
	Then $M$ is free.
\end{corollary}

\begin{proof}
	By Proposition~\ref{prop:min-mult-ARC}, $\pd_R M $ is finite. So Lemma~\ref{nonvanish} yields that $M$ is free.
\end{proof}

\section{Tor rigidity over deformations of rings of minimal multiplicity}

In this section, not only we strengthen \cite[Cor.~6.5]{GP19b}, but also we analyze the rigidity of an arbitrary module over a deformation of a CM local ring of minimal multiplicity. Using Proposition~\ref{prop:vanishing-Tor-and-finite-pd} in place of \cite[Thm.~4.3]{GP19b} in the proof of \cite[Thm.~6.4]{GP19b}, one should get the following refinement of \cite[Thm.~6.4]{GP19b}.

\begin{theorem}\label{thm:defor-min-mult-Tor-vanishing}
	Suppose $R$ is CM of minimal multiplicity. Let $ S = R/(f_1,\dots,f_c)R $, where $ f_1,\dots,f_c \in \fm $ is an $R$-regular sequence.  Let $M$ and $N$ be $S$-modules such that $M$ is MCM. Then, the following are equivalent:	
	\begin{enumerate}[\rm (1)]
		\item $ \Tor_i^S(M, N) = 0 $ for some $ (d+c+1) $ consecutive values of $ i \ge 2 $. 
		\item $ \Tor_i^S(M, N) = 0 $ for all $ i \ge 1 $. 
	\end{enumerate}	
	Moreover, if this holds true, then either $ \pd_R M < \infty $ or $ \pd_R N < \infty $.
\end{theorem}

\begin{proof}
	Clearly (2)$\implies$(1). So, we prove that (1)$\implies$(2), and the statement about projective dimension. We use induction on $c$. The base case that $c=0$ follows from Proposition~\ref{prop:vanishing-Tor-and-finite-pd}. For clarity, we also do the $c=1$ case:
	
	In this case, $S=R/(f_1)$. Since 	$\Tor_i^S(M,N)=0$ for some $(d+1+1)$ consecutive values of $i\ge 2$, in view of \cite[Lem.~2.7]{GP19b}, $\Tor^{R}_i(M,N)=0$ for some $(d+1)$-consecutive values of $ i \ge 2+1 = 3 $. Note that $ \depth_R(M) = \depth_S(M) = \dim(S) = \dim(R)-1 $. So, $ \Omega_1^R(M) $ is an MCM $R$-module, and $ \Tor^R_{i}(\Omega_1^R(M), N) = \Tor^R_{i+1}(M,N)=0 $ for $ (d+1) $-consecutive values of $i\ge 2$. Since $\Omega_1^R(M)$ is MCM, by Proposition~\ref{prop:vanishing-Tor-and-finite-pd}, either $ \pd_R \Omega_1^R(M) < \infty $ or $\pd_R N<\infty$, i.e., either $\pd_R M<\infty$ or $\pd_R N<\infty$. Proposition~\ref{prop:vanishing-Tor-and-finite-pd} also yields that $\Tor^R_{i \ge 1}(\Omega_1^R(M), N)=0$, hence $\Tor^R_{i\ge 2}(M,N)=0$. By \cite[Lem.~2.7]{GP19b}, it follows that $\Tor^S_{i+2}(M,N)\cong \Tor^S_i(M,N)$ for all $ i \ge 1 $. Since, by assumption, $\Tor^S_i(M,N)=0$ for some $(d+2)$-consecutive values of $i\ge 2$, we get that $\Tor^S_i(M,N)=0$ for all $i \ge 1$.
	
	For the inductive step, let $ c \ge 1 $. Set $R' := R/(f_1,\dots,f_{c-1})$, so that $ S = R'/(f_c) $. Since $ \Tor_i^S(M,N) = 0 $ for some $(d+c+1)$ consecutive values of $i\ge 2$, in view of \cite[Lem.~2.7]{GP19b}, we obtain $\Tor^{R'}_i(M,N)=0$ for some $(d+c)$ consecutive values of $i\ge 3$. Note that $ \depth_{R'} M = \depth_S M = \dim S = \dim(R')-1 $. So, $ \Omega_1^{R'}(M) $ is an MCM $R'$-module, and $ \Tor^{R'}_i(\Omega_1^{R'}(M), N) = \Tor^{R'}_{i+1}(M,N) = 0 $ for $(d+c)$ consecutive values of $ i \ge 2 $. Then, by induction hypothesis, $ \Tor^{R'}_i(\Omega_1^{R'}(M), N) = 0 $ for all $ i \ge 1 $, and $ \pd_R (\Omega_1^{R'}(M)) < \infty $ or $ \pd_{R} N < \infty $. Since $\pd_{R} R'<\infty$, this means either $ \pd_R M < \infty $ or $ \pd_R N < \infty $. It remains to show that $ \Tor^S_i(M,N) = 0 $ for all $ i \ge 1 $. Firstly, $ \Tor^{R'}_i(\Omega_1^{R'}(M), N) = 0 $ for all $ i \ge 1 $ yields that $ \Tor^{R'}_{j}(M,N) = 0 $ for all $ j \ge 2 $.  Hence, by \cite[Lem.~2.7]{GP19b}, $ \Tor^S_{i+2}(M,N) \cong \Tor^S_i(M, N) $ for all $ i \ge 1 $. Therefore, since by assumption, $ \Tor^S_i(M, N) = 0 $ for some $ d+c+1 \,(\ge 2) $ consecutive values of $ i \ge 2 $, so we get $ \Tor^S_i(M,N) = 0 $ for all $ i \ge 1 $.
\end{proof}

The following refines \cite[Cor.~6.5]{GP19b}. Note that the threshold in \cite[Cor.~6.5]{GP19b} is $2\dim(S) - \depth(M) - \depth(N)+c+2$, while in our next result, the threshold is $\dim(S)-\depth(M)+2$, particularly, it is $c$ independent.

\begin{corollary}\label{deformcor} 
	Suppose $R$ is CM of minimal multiplicity. Let $ S = R/(f_1,\dots,f_c)R $, where $ f_1,\dots,f_c \in \fm $ is an $R$-regular sequence. Let $M$ and $N$ be $S$-modules. Set $i_0:=\dim(S)-\depth(M)$. Then, the following are equivalent:	
	\begin{enumerate}[\rm (1)]
		\item  $\Tor_i^S(M,N)=0$ for some $(d+c+1)$ consecutive values of $i\ge i_0+2$. 
		\item $\Tor_i^S(M,N)=0$ for all $i\ge i_0+1$.  
	\end{enumerate}	
	Moreover, if this holds true, then  either $ \pd_R M < \infty $ or $ \pd_R N < \infty $.
\end{corollary}  

\begin{proof}
	It is enough to prove that (1) implies (2) and the statement about projective dimension. Note that $ \Omega^S_{i_0}(M) $ is an MCM $ S $-module. The condition (1) implies that $ \Tor^S_{i}(\Omega^S_{i_0}(M), N) = 0 $ for $ (d+c+1) $ consecutive values of $ i \ge 2 $. Hence, by Theorem~\ref{thm:defor-min-mult-Tor-vanishing}, it follows that $ \Tor^S_i(\Omega^S_{i_0}(M), N) = 0 $ for all $i \ge 1$. Therefore $\Tor_i^S(M,N)=0$ for all $i\ge i_0+1$. Moreover either $ \pd_R \Omega^S_{i_0}(M) < \infty $ or $ \pd_R N < \infty $, hence either $ \pd_R M < \infty $ (as $ \pd_R S < \infty$) or $ \pd_R N < \infty $.
\end{proof}

If $N$ is embedded in an $S$-module of finite projective dimension, then the threshold $i_0+2$ in Corollary~\ref{deformcor} can be further improved to $i_0+1$. In Lemma~\ref{projembed}, a few classes of such modules $N$ are described.
	
\begin{corollary}\label{gdeform}
	Suppose $R$ is CM of minimal multiplicity. Let $ S = R/(f_1,\dots,f_c)R $, where $ f_1,\dots,f_c \in \fm $ is an $R$-regular sequence.  Let $M$ and $N$ be $S$-modules. Set $i_0:=\dim(S)-\depth(M)$. Suppose $N$ can be embedded in an $S$-module of finite projective dimension over $S$. Then, the following are equivalent:
	\begin{enumerate}[\rm (1)]
		\item $ \Tor_i^S(M,N) = 0 $ for some $ (d+c+1) $ consecutive values of $i\ge i_0+1$. 
		\item $ \Tor_i^S(M,N) = 0 $ for all $ i \ge i_0+1 $.   
	\end{enumerate}	
	Moreover, if this holds true, then  either $\pd_R M<\infty$ or $\pd_R N<\infty$.
\end{corollary}

\begin{proof}
	By the given hypotheses, there is an exact sequence $(\dagger)$: $0\to N \to Y \to X \to 0$ of $S$-modules with $ \pd_S Y < \infty $. Applying $ \Omega_{i_0}^S(M) \otimes_S (-) $ to $(\dagger)$, since $ \Tor^S_{i \ge 1}(\Omega^S_{i_0}(M), Y)=0$ (by \cite[Lem.~2.2]{Yo}), it follows that
	\begin{center}
		$ \Tor^S_{i+1}(\Omega^S_{i_0}(M), X) \cong \Tor^S_i(\Omega^S_{i_0}(M), N) $ for all $i \ge 1$.
	\end{center}
	So $ \Tor^S_{i+i_0+1}(M, X) \cong \Tor^S_{i+i_0}(M,N) $ for all $i \ge 1$, i.e.,
	\begin{equation}\label{Tor-iso}
		\Tor^S_{j+1}(M, X) \cong \Tor^S_j(M, N) \mbox{ for all } j \ge i_0+1.
	\end{equation}
	Suppose the condition (1) holds true. Then \eqref{Tor-iso} yields that $ \Tor^S_{i+1}(M,X) = 0 $ for some $ (d+c+1) $ consecutive values of $ i+1 \ge i_0+2 $. Therefore, by Corollary~\ref{deformcor}, $\Tor_i^S(M,X)=0$ for all $i\ge i_0+1$. Hence, by \eqref{Tor-iso}, $\Tor_i^S(M,N)=0$ for all $i\ge i_0+1$, which is nothing but the condition (2). Under these equivalent conditions, by Corollary~\ref{deformcor}, one concludes that either $\pd_R M<\infty$ or $\pd_R N<\infty$.
\end{proof}    

Now we record a lemma which provides some classes of modules which can be embedded in modules of finite projective dimension. These support Corollary~\ref{gdeform}.  For this, we first record a preliminary observation.

\begin{lemma}\label{funct}
	The following statements hold true.
\begin{enumerate}[\rm(1)]    
    \item
    Let $F:\Mod R\to \Mod R$ be a contravariant functor, and let $\Phi: \Id \to F\circ F$ be a natural transformation, where $\Id: \Mod R \to \Mod R$ is the identity functor. Let $N$ be an $R$-module for which $\Phi_N:N\to \left(F\circ F\right)(N)$ is an isomorphism. Then, $F_{M,N}:\Hom_R(M,N)\to \Hom_R(F(N),F(M))$ is injective for all modules $M$.    
    \item
    Suppose that $R$ is CM which admits a canonical module $\omega$. Let $M$ and $N$ be $R$-modules such that $N$ is MCM. Then, the natural map $\Hom_R(M,N)\to \Hom_R(N^{\dagger},M^{\dagger})$ is injective, where $(-)^{\dagger} := \Hom_R(-,\omega)$.
\end{enumerate}
\end{lemma}

\begin{proof}
	(1) For every module $M$, and $f\in \Hom_R(M,N)$, we have the following commutative diagram:
$$\begin{tikzcd}
M \arrow[r, "f"] \arrow[d, "\Phi_M"']      & N \arrow[d, "\Phi_N"] \\
(F \circ F)(M) \arrow[r, "(F\circ F)(f)"'] & (F \circ F)(N),  
\end{tikzcd}$$ 
which yields  $\Phi_N^{-1}\circ(F\circ F)(f) \circ \Phi_M= f$. Hence, if $F(f)=0$, then $(F\circ F)(f)=F\left(F(f)\right)=0$, and hence $f=0$. Thus, for every module $M$, $F_{M,N}:\Hom_R(M,N)\to \Hom_R(F(N),F(M))$ is injective.
	
	(2) Apply (1) to the functor $ F(-) := (-)^{\dagger} = \Hom_R(-,\omega) $ and $ \Phi_X : X \to X^{\dagger \dagger} $ being the natural evaluation map for all $X$. Remembering that $\Phi_N$ is an isomorphism whenever  $N$ is MCM (\cite[3.3.10.(d)]{BH93}), we are done. 
\end{proof}  

\begin{lemma}\label{projembed}
	Let $N$ be an $R$-module. Then $N$ can be embedded in a module of finite projective dimension in each of the following cases:
	\begin{enumerate}[\rm(1)]
		\item $ \gdim_R N < \infty $.
		\item $ R $ is CM, and it admits a canonical module $\omega$, and $N\cong \Hom_R(\omega, X)$ for some $R$-module $X$.		
		\item 
  $N\cong \Hom_R(X,\syz M)$ for some $R$-modules $M$ and $X$. 

  \item  $ R $ is CM, and it admits a canonical module, and $N\cong \Hom_R(\left(\syz M\right)^{\dagger},X)$, for some $R$-modules $M$ and $X$ where $\depth M\ge \dim R-1$ and $X$ is MCM.    
	\end{enumerate}
\end{lemma} 

\begin{proof}
	(1) This is shown in \cite[Lem.~2.17]{CFH}. 
	
	(2) By \cite[Thm.~A]{AB89}, there is an embedding $ 0 \to X \to Y $, where $Y$ is an $R$-module of finite injective dimension. This induces an embedding $ 0 \to \Hom_R(\omega, X) \to \Hom_R(\omega, Y) $. So, it is enough to show that $\Hom_R(\omega, Y)$ has finite projective dimension, which holds true as recorded in \cite[Exercise~9.6.5.(b)]{BH93}.
	
(3) Since $\syz M$ embeds in a free $R$-module, hence $\Hom_R(X,\syz M)$ embeds in $\Hom_R(X,R)^{\oplus n}$ for some integer $n\ge 0$. As $\Hom_R(X,R)$ embeds in a free $R$-module, we are done.  
 

 (4) As $\depth M\ge \dim R-1$, so $\syz M$ is MCM, hence $\left(\syz M\right)^{{\dagger}{\dagger}}\cong \syz M$. By Lemma \ref{funct}(2),  $\Hom_R\left(\left(\syz M\right)^{\dagger},X\right)$ embeds in $\Hom_R\left(X^{\dagger}, \left(\syz M\right)^{{\dagger}{\dagger}}\right)\cong \Hom_R\left(X^{\dagger},\syz M\right)$, and we are done by part (3).
\end{proof}      

\begin{remark}
	In view of Lemma~\ref{projembed}.(1), Corollary~\ref{gdeform} holds true when $\gdim_S N<\infty$ (equivalently, $\gdim_R N<\infty$, cf.~\cite[(2.2.8)]{Cr00}).
\end{remark}

Using the Tor vanishing criteria of Corollary \ref{deformcor} and Corollary \ref{gdeform} , an argument similar to Corollary \ref{minrigid} shows the following

\begin{corollary}\label{cor:rigidity-over-deform}
	Suppose $R$ is CM of minimal multiplicity. Let $ S = R/(f_1,\dots,f_c)R $, where $ f_1,\dots,f_c \in \fm $ is an $R$-regular sequence. Let $M$ be an $S$-module with $ i_0 := \dim(S) - \depth(M) $. Then, $M$ is $(d+c+i_0+2)$-$\Tor$-rigid over $S$. 
\end{corollary} 

\begin{proof}
	Let $N$ be an $S$-module. Suppose that there exists an integer $n \ge 1$ such that $ \Tor^S_{n \le i \le n+d+c+i_0+1}(M,N) = 0 $. Then $ \Tor^S_{n+i_0+1 \le i \le n+d+c+i_0+1}(M,N) = 0 $, and these are $ d+c+1 $ consecutive values of $i \ge n+i_0+1 \ge i_0+2 $. Hence, by Corollary~\ref{deformcor}, we get $\Tor^S_{i\ge i_0+1}(M,N)=0$, and thus $\Tor^S_{i\ge n+d+c+i_0+1}(M,N)=0$.
\end{proof}   

\section{The Auslander-Reiten conjecture over deformations of rings of minimal multiplicity}\label{sec:Ext-persistency-ARC}

In this section, we study a form of Ext-persistency of a deformation of a CM local ring of minimal multiplicity. We also prove that the Auslander-Reiten conjecture (\ref{ARC}) holds true over such rings. For these, a few lemmas are needed.

\begin{lemma}\label{lem:Tor-vanishing-reg-seq}
	Let $ {\bf f} := f_1,\dots,f_c $ be an $R$-regular sequence. Let $N$ be an $R$-module such that $\Tor_{ 1}^R(N,R/({\bf f})) = 0$. Then the sequence $ {\bf f} $ is also $N$-regular.
\end{lemma}

\begin{proof}
	We use induction on $c$. First assume that $c=1$. Considering the exact sequence $0\to R \xrightarrow{f_1 \cdot} R \to R/(f_1)\to 0$, and applying $N\otimes_R (-)$, we get a short exact sequence $0\to N \xrightarrow{f_1 \cdot} N \to N/(f_1)N\to 0$. Thus $f_1$ is $N$-regular. Next assume that $c\ge 2$. Set $ R' := R/(f_1,\dots,f_{c-1}) $ and $N' := N \otimes_R R'$. From the long exact sequence of Tor modules induced by the short exact sequence $0\to R' \xrightarrow{f_c \cdot} R' \to R/({\bf f})\to 0$,
	we get that $\Tor^R_1(N,R')\xrightarrow{f_c \cdot} \Tor^R_1(N,R')\to 0$, hence $\Tor^R_1(N,R')=f_c\cdot\Tor^R_1(N,R')$. By the Nakayama's Lemma, it follows that $\Tor_{1}^R(N,R') = 0$. Therefore, by the induction hypothesis, $f_1,\dots,f_{c-1}$ is $N$-regular. Hence
	\[
		\Tor^{R'}_{1}\big( N', R'/f_cR' \big) \cong \Tor^R_{1}(N, R/({\bf f})) = 0, \;\mbox{ cf. \cite[p.~140, Lem.~2]{Mat86}}.
	\]
	So, by the base case, $f_c$ is $N'$-regular. Thus the sequence $ {\bf f} $ is $N$-regular.
\end{proof}

We say that $R$ is `weakly Ext-persistent' if the following condition is satisfied: For every $R$-module $M$, with $\dim(R)-\depth(M)\le 1$, the vanishing of $\Ext^{\ge 2}_R(M,M)$ implies that either $\pd_R M<\infty$ or $\id_R M<\infty$.   

\begin{lemma}\label{extdeform}
	Let $x\in \fm$ be an $R$-regular element. If $R$ is weakly $\Ext$-persistent, then so is $R/(x)$.  
\end{lemma} 

\begin{proof}
	Let $M$ be an $R/(x)$-module such that $\dim(R/(x))-\depth(M)\le 1$ and $\Ext^{\ge 2}_{R/(x)}(M,M)=0$. By \cite[Prop.~1.7]{ADS93}, there exists an $R$-module $N$ such that $M \cong N/xN$ and $\Tor^R_{\ge 1}(N,R/(x))=0$. Considering the exact sequence
	$0\to R \xrightarrow{x \cdot} R \to R/(x)\to 0$, and applying $N\otimes_R (-)$, we get a short exact sequence $(\alpha)$: $0\to N \xrightarrow{x\cdot} N \to M\to 0$. Hence $x$ is $N$-regular. Since $\Ext^{\ge 2}_{R/(x)}(M,M)=0$, by \cite[3.1.16]{BH93}, it follows that $\Ext^{\ge 3}_{R}(M,N)=0$.
	Therefore, applying $\Hom_R(-,N)$ to $(\alpha)$, for every $i\ge 2$, there is an exact sequence
	\begin{center}
		$\Ext^i_R(N,N)\xrightarrow{x\cdot}\Ext^i_R(N,N)\to \Ext^{i+1}_R(M,N)=0$.
	\end{center}
	Hence, by the Nakayama's Lemma, $\Ext^i_R(N,N)=0$ for all $i\ge 2$. Note that $\dim(R)-\depth(N)=\dim(R/(x))+1-\depth(M)-1=\dim(R/(x))-\depth(M)\le 1$. Therefore, since $R$ is weakly-Ext persistent, either $\pd_R N<\infty$ or $\id_R N<\infty$.
	It follows that either $\pd_{R/(x)} M<\infty$ or $\id_{R/(x)} M<\infty$; see, e.g., \cite[1.3.5 and 3.1.15]{BH93}. Thus $R/(x)$ is weakly-Ext persistent.    
\end{proof} 

\begin{theorem}\label{thm:defor-Ext-persistency}
	Suppose $R$ is CM of minimal multiplicity. Let $ S = R/(f_1,\dots,f_c)R $, where $ f_1,\dots,f_c $ is an $R$-regular sequence. Let $M$ be an $S$-module such that $ i_0 := \dim(S)-\depth(M)\le 1$ and $\Ext^i_S(M,M)=0$ for some $(d+c+1)$ consecutive values of $i\ge i_0+2$. Then either $\pd_S M < \infty $ or $\id_S M<\infty$.
\end{theorem}

\begin{proof}
	By \cite[Cor.~6.3]{GP19b}, we get that $\Ext^i_S(M,M)=0$ for all $i \ge i_0+1$. In particular, $\Ext^{\ge 2}_S(M,M)=0$ as $i_0 \le 1$. The proof follows if we show that $S$ is weakly-Ext persistent. By Lemma~\ref{extdeform} and induction on $c$, it is enough to show that $R$ is weakly-Ext persistent. But that easily follows from Proposition~\ref{prop:improvment-Ghosh-Puthen}(2). Indeed, let $M$ be an $R$-module such that $\Ext^{\gg 0}_R(M,M)=0$. If $\pd_R M=\infty$, then $\syz_{d+1}^R(M)$ is an Ulrich $R$-module. Also, $\Ext^{\gg 0}_R(\syz_{d+1}^R(M), M)=0$, and hence it follows from Proposition~\ref{prop:improvment-Ghosh-Puthen}.(2) that $\id_R M<\infty$.  
\end{proof} 

\begin{example}
	Over a local complete intersection ring $S$ of dimension $1$, the vanishing of $\Ext^i_S(M,N)=0$ for all $i \gg 0$ does not necessarily imply that either $\pd_S M < \infty $ or $\id_S N<\infty$. Set $S := k[[x,y,z]]/(xz-y^2, xy - z^2)$ with $k$ a field, $M := S/(x,y)$ and $N := S/(x, z)$. It is shown in \cite[4.2]{Jor99} that $\Tor^S_i(M,N)=0$ for all $i \ge 2$, but $\pd_S M = \infty $ and $\pd_S N = \infty$. With the same method, it can be observed that $\Ext^i_S(M,N)=0$ for all $i \ge 2$, but $\pd_S M = \infty $ and $\id_S N = \infty$.
\end{example}


	\begin{remark}
		Using the methods of \cite[Sec.~6]{AINS20} (in particular, \cite[Prop.~6.7]{AINS20}), the condition $\dim(S)-\depth(M)\le 1$ can be dropped from Theorem~\ref{thm:defor-Ext-persistency}. However, \cite[Prop.~6.7]{AINS20} is quite non-trivial and uses the detailed structure of total Ext-algebra. Compared to that, our method, although gives weaker result, are elementary, and only depends on \cite[Prop.~1.7]{ADS93}.
	\end{remark}

Now we are in a position to show that the Auslander-Reiten conjecture holds true over a deformation of a CM local ring of minimal multiplicity.

\begin{conjecture}[Auslander-Reiten, \cite{AR75}]\label{ARC}
	For an $R$-module $M$, if $ \Ext^i_R(M, M \oplus R) = 0 $ for all $ i > 0 $, then $M$ is projective.
\end{conjecture}

\begin{theorem}\label{thm:ARC-over-deform}
	Suppose $R$ is CM of minimal multiplicity. Let $ S = R/(f_1,\dots,f_c)R $, where $ f_1,\dots,f_c $ is an $R$-regular sequence. Let $M$ be an $S$-module. Set $i_0 := \dim(S)-\depth(M)$. Let $n\ge i_0+2$ be an integer. Then
	\begin{enumerate}[\rm (1)]
		\item If $\Ext^i_S(M,M)=\Ext_S^j(M,S)=0$ for all $n \le i \le n+(d+c)$ and $n \le j \le n + i_0 + (d+c)$, then $\pd_S M < \infty$.
		\item If $\Ext^i_S(M,M)=\Ext_S^j(M,S)=0$ for all $i_0+2 \le i \le i_0+(d+c)+2$ and $1 \le j \le 2i_0 + (d+c)+2$, then $M$ is free.
	\end{enumerate}
\end{theorem}

\begin{proof}
	(1) We prove the statement by way of contradiction. If possible, assume that $\pd_S M = \infty$. By Lemma~\ref{lem:Ext-vanishing},
	\begin{center}
		$\Ext^j_S(\Omega_{i_0}^S(M), \Omega_{i_0}^S(M))=0 $ for all  $ n\le j \le n+(d+c) $.
	\end{center}
	Note that $\Ext_S^j(\Omega_{i_0}^S(M), S) \cong \Ext_S^{j+i_0}(M,S)=0$ for all $n - i_0 \le j \le n + (d+c)$. Since $ \Omega_{i_0}^S(M) $ is an MCM $S$-module, replacing $M$ by $\Omega_{i_0}^S(M)$, we may assume that $M$ is MCM. Therefore, by virtue of \cite[Thm.~6.2]{GP19b},
	\begin{equation}
		\Ext^i_S(M,M)=\Ext_S^i(M,S)=0 \mbox{ for all } i \ge 1.
	\end{equation}
	Moreover, \cite[Thm.~6.2]{GP19b} yields that
	\begin{equation}\label{either-or-pd-id}
	\mbox{either $\pd_R M < \infty$ or $\id_R S < \infty$.}
	\end{equation}
	Since $\pd_S M = \infty$, in view of Theorem~\ref{thm:defor-Ext-persistency}, we have
	\begin{equation}\label{id-finite}
	\id_S M < \infty.
	\end{equation}
	Set $ {\bf f} := f_1,\dots,f_c $. Since $ \Ext^2_S(M,M)=0$, by \cite[Prop.~1.7]{ADS93}, there exists an $R$-module $N$ such that $M \cong N/({\bf f})N$ and $\Tor_{\ge 1}^R(N,R/({\bf f})) = 0$. Then Lemma~\ref{lem:Tor-vanishing-reg-seq} yields that $ {\bf f} $ is also $N$-regular. It follows that
	\begin{equation*}
	\pd_R M = \pd_R N/({\bf f})N = \pd_R N + c = \pd_{R/({\bf f})} N/({\bf f})N +c = \pd_S M +c = \infty,
	\end{equation*}
	see, e.g., \cite[1.3.6 and 1.3.5]{BH93}. Therefore, by \eqref{either-or-pd-id}, $\id_R S < \infty$, which implies that $\id_R R < \infty$, i.e., $R$ is Gorenstein, and hence $S$ is Gorenstein. So \eqref{id-finite} yields that $\pd_S M < \infty$ (cf.~\cite[3.1.25]{BH93}), which contradicts the assumtion that $\pd_S M = \infty$. Therefore $\pd_S M < \infty$.
	
	(2) It follows from (1) that $\pd_S M < \infty$. Hence, by Lemma~\ref{nonvanish}, $M$ is free.
\end{proof}

\begin{remark}
	The Auslander-Reiten conjecture is known to hold true in many particular cases. See \cite[Cor.~1.3]{GT21} and the preceding paragraph for a short survey on this conjecture.
\end{remark}

\section*{Acknowledgments}
The authors thank the anonymous referee for various suggestions that improve the presentation of the article. Ghosh was supported by Start-up Research Grant (SRG) from SERB, DST, Govt.~of India with the Grant No SRG/2020/000597.

\end{document}